\date{} 
\title{A contour line of the \\ continuum 
Gaussian free field}
\author{Oded Schramm \and Scott Sheffield}
\documentclass[12pt]{article}
\usepackage{amsmath}
\usepackage{amssymb}
\usepackage{amsthm}
\usepackage{amsfonts}
\usepackage{graphicx}
\usepackage{url}
\newif\ifdraft
\drafttrue \long\def\note#1/{\ifdraft{\marginpar{{$\Longleftarrow$}} \bf [#1] }\fi}
\long\def\comment#1{} \long\def\old#1{}

\numberwithin{equation}{section}
\numberwithin{figure}{section}

\newtheorem{theorem}{Theorem}
\numberwithin{theorem}{section}
\newtheorem{corollary}[theorem]{Corollary}
\newtheorem{lemma}[theorem]{Lemma}
\newtheorem{proposition}[theorem]{Proposition}

\theoremstyle{remark}
\theoremstyle{remark}\newtheorem{remark}[theorem]{Remark}
\def\eref#1{(\ref{#1})}
\let\qqed=\qed
\def\QED{\qqed\medskip}
\let\qed=\QED

\newcommand{\R}{\mathbb{R}}
\newcommand{\C}{\mathbb{C}}

\def\H{\mathbb{H}}
\def\D{\mathbb{D}}
\def\U{\mathbb{U}}
\def\diam{\mathop{\mathrm{diam}}}

\def\dist{\mathop{\mathrm{dist}}}

\def\Im{{\rm Im}\,}
\def\Re{{\rm Re}\,}
\def\ccup{\check{\cup}}
\def\Harm{{\mathrm{Harm}}}
\def\Supp{{\mathrm{Supp}}}

\def\SLEkk#1/{$\mathrm{SLE}(#1)$}
\def\SLEr#1/{$\mathrm{SLE(\kappa;#1)}$}
\def\SLEkr#1;#2/{$\mathrm{SLE(#1;#2)}$}
\def\SLEk/{\SLEkk{\kappa}/}
\def\SLEtwo/{\SLEkk2/}
\def\SLE/{$\mathrm{SLE}$}
\def\standardg/{standard Gaussian}

\def\Ito/{It\^o}
\def \eps {\epsilon}
\def \P {{\bf P}}
\def\md{\mid}
\def\Bb#1#2{{\def\md{\bigm| }#1\bigl[#2\bigr]}}
\def\BB#1#2{{\def\md{\Bigm| }#1\Bigl[#2\Bigr]}}
\def\Bs#1#2{{\def\md{\mid}#1[#2]}}
\def\Pb{\Bb\P}
\def\Eb{\Bb\E}
\def\PB{\BB\P}

\def\Ps{\Bs\P}

\def \p {{\partial}}
\def \E {{\bf E}}

\def\closure{\overline}

\def \proof {{ \medbreak \noindent {\bf Proof.} }}
\def\proofof#1{{ \medbreak \noindent {\bf Proof of #1.} }}

\def\setminus{\smallsetminus}

\def\capacity{\mathrm{cap}_\infty}

\def\inr#1{\mathrm{rad}_{#1}}

\def\WW{W^*}
\def\WW_#1{{\tilde W}_{s_#1}}

\def\pphi{\psi}

\def\TG{TG}

\def\dhaus{{d}_{\text{HAUS}}}
\def\dstrong{{d}_{\text{STRONG}}}

\def\rr{{r_D}}
\def\rrr#1{{r_{#1}}}

\def\hg{{\widehat \gamma}}

\def\vv{v_0}

\def\density{{\mathfrak p}}
\def\condexp{\mathbf C}

\def\LSWlesl{math.PR/0112234}
\def\LSWlesl{MR2044671}

\def\Dudley{MR91g:60001}
\def\Dudley{MR1932358}

\def\RSsle{MR2153402}

\def\PommeBDRY{MR95b:30008}

\def\ProtterBook{MR1037262}

\def\GFFSurvey{math.PR/0312099}
\def\Janson{MR1474726}

\def\DGFFpaper{MR2486487}
\def\Gross{MR0212152}

\def\noopsort#1{}

\begin{document}
\maketitle

\begin{abstract}
Consider an instance $h$ of the Gaussian free field on a simply connected planar
domain $D$ with boundary conditions $-\lambda$ on one boundary arc
and $\lambda$ on the complementary arc, where $\lambda$ is the
special constant $\sqrt{\pi/8}$. We argue that even though $h$ is
defined only as a random distribution, and not as a function, it has
a well-defined zero level line $\gamma$ connecting the endpoints of
these arcs, and the law of $\gamma$ is \SLEkk4/. We construct
$\gamma$ in two ways: as the limit of the chordal zero contour lines
of the projections of $h$ onto certain spaces of piecewise linear
functions, and as the only path-valued function on the space of
distributions with a natural Markov property.

We also show that, as a function of $h$, $\gamma$ is ``local'' (it does not change when $h$ is modified away from $\gamma$)
and derive some general properties of local sets.

\end{abstract}

\newpage
\tableofcontents
\newpage
\section{Introduction} \label{introsection}
\subsection{Overview}
The two dimensional Gaussian free field (GFF) is an object of central importance in mathematics and physics.
It appears frequently as a model for random surfaces and height interfaces and as a tool
for studying two-dimensional statistical physics models that are not obviously random surfaces (e.g., Ising and
Potts models, $O(n)$ loop models).  It appears in random matrix theory (e.g.\ \cite{MR2361453}), as a random electrostatic potential in Coulomb gas theory, and as the starting point for many constructions in conformal field theory and string theory.  It is related to the Schramm-Loewner evolution (SLE) in a range of ways \cite{S05, MR2525778, 2010arXiv1006.1853I, 2010JSP...140....1H} (see Section \ref{ss.mainresults}) and it represents the logarithmic conformal metric distortion in critical Liouville quantum gravity (see details and references in
\cite{2008arXiv0808.1560D}).

This paper is a sequel to another paper by the current authors
\cite{\DGFFpaper}, which contains a much more detailed overview of the topics mentioned above, and many more references.
That paper studied the discrete Gaussian free field (DGFF), which is a random function on a graph that (when defined on increasingly fine lattices) has the GFF as a scaling limit.  The authors showed in \cite{\DGFFpaper} that a certain level line of the DGFF has \SLEkk 4/ as a scaling limit.

More precisely, when
one defines the DGFF on a planar lattice graph with triangular faces (interpolating the DGFF linearly on each triangle to produce
a continuous function on a domain in $\R^2$), with positive boundary conditions on one boundary arc and negative boundary
conditions on the complementary arc, the zero chordal contour line connecting the two endpoints of these arcs converges in law
(as the lattice size gets finer) to a variant of \SLEkk 4/.   In particular, there is a special constant $\lambda>0$ such that if
boundary conditions are set to $\pm \lambda$ on the two boundary arcs, then the zero chordal contour
line converges to \SLEkk 4/ itself as the lattice size tends to zero.
The exact value of $\lambda$ was not
determined in \cite{\DGFFpaper}, but we will determine it here.
Specifically, we will show that if the DGFF is scaled in such a way
that its fine-mesh limit is the continuum Gaussian free field (GFF), defined by the Dirichlet norm $\int |\nabla \phi|^2$,
then $\lambda = \sqrt{\pi/8}$.  (Another common convention is to take $(2\pi)^{-1}\int |\nabla \phi|^2$ for the Dirichlet norm.  Including the factor $(2\pi)^{-1}$ is equivalent to multiplying the GFF by $\sqrt{2\pi}$, which would make $\lambda = \pi/2$.)

It was observed in \cite{\DGFFpaper} that one can project an
instance $h$ of the GFF on a simply connected planar domain onto a sequence of
subspaces to obtain a sequence of successively closer continuous and
piecewise linear
approximations to $h$, each of
which has the law of a discrete Gaussian free field. Although $h$ is
defined only as a distribution, and not as a continuous function,
one might hope to define the ``contour lines'' of $h$ as the limits
of the contour lines of these approximations. The work in
\cite{\DGFFpaper} implies that (when the boundary conditions are
$\pm \lambda$ on complementary arcs) the zero chordal contour lines
of the approximations converge {\em in law} to \SLEkk4/. Our goal is
to strengthen these results and show that these contour lines converge {\em in
probability} to a path-valued function $\gamma$ of $h$ whose law is
\SLEkk 4/.  We will also characterize $\gamma$ directly by showing
that it is the unique path-valued function of $h$ with a certain
Markov property.   In Section \ref{s.local}, we also show that, as a function of
$h$, $\gamma$ is ``local.''  Very roughly speaking, this means that $\gamma$ can be determined
without observing the values of $h$ at points away from $\gamma$ --- and thus, $\gamma$ is not
changed if $h$ is modified away from $\gamma$.  We will also give a discrete definition of local
and show that both discrete and continuum local sets have
nice properties.

The reader may consult \cite{\DGFFpaper} for background and a historical
introduction to the contour line problem for discrete and
continuum Gaussian free fields.  We will forego such an introduction here and proceed
directly to the notation and the main results.

\medbreak {\noindent\bf Acknowledgments.} We wish to thank Vincent
Beffara, Richard Kenyon, Jan\'{e} Kondev, Julien Dub\'edat, and
David Wilson for inspiring and useful conversations.  In particular, Kenyon
and the first author worked out a coupling of the GFF and \SLEkk8/ early on, which is quite similar
to the coupling between the GFF and \SLEkk4/ presented here.  We also thank Jason Miller for helpful
comments on a draft of this paper.

This paper was mostly finished well before the tragic 2008 death of the first author, Oded Schramm.  The results were originally part of our first draft of \cite{\DGFFpaper} (begun in 2003) and were separated once it became apparent that \cite{\DGFFpaper} was becoming quite long for a journal paper.  We presented them informally in slides and online talks (e.g.\ \cite{S05}) and mostly finished the writing together.  The completion of this project is a somewhat melancholy occasion for the second author, as it concludes a long and enjoyable collaboration with an inspiring mathematician and a wonderful friend.  It is also another occasion to celebrate the memory of Oded Schramm.

\subsection{Notation and definitions} \label{ss.notation}

Let $D$ be a planar domain (i.e.,
a connected open subset of $\R^2$, which we sometimes identify with the complex
plane $\C$).  Assume further that $D$ is a subset of a simply connected domain that
is not all of $\C$.  (In particular, this implies that
$D$ can be mapped conformally to a subset of the unit disc.)

For each $x \in D$, the point in $\partial D$ at which a Brownian motion started at $x$
first exits $D$ is a random variable whose law
we denote by $\nu_x$ and call the {\bf harmonic measure of $\partial D$ viewed from
$x$}.  It is not hard to show that for
any $y \in D$ the measures $\nu_x$ and $\nu_y$ are absolutely continuous
with respect to one another, with a Radon Nikodym derivative bounded between two positive constants.
In particular, this implies that if a function $f_\partial : \partial D \to \R$ lies in $L^1(\nu_x)$
then it lies in $L^1(\nu_y)$ for any $y \in D$.
We refer to a function with this property as a {\bf harmonically integrable} function of $\partial D$.
If $f_\partial:\partial D \to \R$ is harmonically integrable, then the function $$f(x) = \int f_\partial(y)
d\nu_x(y),$$ defined on $D$, is called the {\bf harmonic extension} of $f_\partial$ to $D$.

Given a planar domain $D$ and $f,g \in L^2(D)$, we denote by $(f,g)$ the
inner product $\int f(x)g(x)  dx$
(where $dx$ is Lebesgue measure on $D$).  Let $H_s(D)$ denote the space of real-valued $C^\infty$,
compactly supported functions on $D$.
We also write $H(D)$ (a.k.a., $H^1_0(D)$ --- a Sobolev space) for
the Hilbert space completion of $H_s(D)$ under the {\bf Dirichlet
inner product} $$(f, g)_{\nabla} := \int_{D}\nabla f \cdot \nabla
g\;dx.$$ We write $\|f\| := (f,f)^{1/2}$ and $\|f\|_\nabla := (f,f)_\nabla^{1/2}$.
If $f,g \in H_s(D)$, then integration by parts gives $(f,g)_\nabla = (f, -\Delta g)$.

Let $\phi$ be a conformal map from $D$ to another domain $D'$.  Then an elementary change of
variables calculation shows that $$\int_{D'} \nabla (f_1 \circ \phi^{-1} ) \cdot
\nabla (f_2 \circ \phi^{-1} )\,dx = \int_{D} (\nabla f_1 \cdot \nabla f_2)\,dx.$$
In other words, the Dirichlet inner product is invariant
under conformal transformations.

It is conventional to use $H_s(D)$ as a space of test functions.  This space is a topological vector space in which the topology is
defined so that $\phi_k \to 0$ in $H_s(D)$ if and only if
there is a compact
set on which all of the $\phi_k$ are supported and
the $m$th derivative of $\phi_k$ converges uniformly to zero for each integer $m \geq 1$.

A {\bf distribution} on $D$ is a continuous linear
functional on $H_s(D)$.
Since $H_s(D) \subset L^2(D)$, we
may view every $h \in L^2(D)$ as a distribution
$\density\mapsto (h,\density)$.
We will frequently abuse notation and use $h$ --- or more precisely the map denoted by $\density \to (h,\density)$ --- to represent a general distribution (which is a functional of $\density$), even though $h$ may not correspond to an element in $L^2(D)$.
We define partial derivatives and integrals of distributions in the
usual way (via integration by parts), i.e.,
$(\frac{\partial}{\partial x} h, \density):= -(h, \frac{\partial}{\partial x}
\density)$; in particular, if $h$ is a distribution then
$\Delta h$ is a distribution defined by $(\Delta h, \density):= (h, \Delta \density)$.
When $h$ is a distribution and $g \in H_s(D)$,
we also write $$(h, g)_\nabla := (-\Delta h, g) = (h, -\Delta g).$$

When $x \in D$ is fixed, we let $\tilde G_x(y)$ be the
harmonic extension to $D$ of the function on $\partial D$
given by $(2\,\pi)^{-1}\log|y-x|$.  (It is not hard to see that this
function is harmonically integrable.)
Then {\bf Green's function in the domain $D$} is defined by
$$G(x,y) = (2\,\pi)^{-1} \log|y-x| - \tilde G_x(y).$$
When $x \in D$ is fixed, Green's function may be viewed as a distributional
solution of $\Delta G(x,\cdot)= -\delta_x(\cdot)$ with zero boundary
conditions.  Note, however, that under the above definition
it is not generally true that $G(x,\cdot)$ extends
continuously to a function on $\overline D$ that vanishes on $\partial D$ (e.g.,
if $\partial D$ contains isolated points,
then $G(x,\cdot)$ need not tend to zero at those points) although
this will be the case if $D$ is simply connected.
For any $\density \in H_s(D)$, we write $-\Delta^{-1} \density$
for the function $\int_D G(\cdot ,y)\,\density (y)\,dy$.  This is a $C^\infty$ (though not necessarily compactly supported)
function in $D$ whose Laplacian is $-\density$.



If $f_1 = - \Delta^{-1} \density_1$ and $f_2 = -\Delta^{-1} \density_2$, then integration by parts gives
$(f_1, f_2)_\nabla = (\density_1,-\Delta^{-1} \density_2).$  By the definition of $-\Delta^{-1} \density_2$ above, the
latter expression may be rewritten as
\begin{equation}\label{e.greencovariance} \int_{D \times D} \density_1(x)\,
\density_2(y)\, G(x,y) \, dx \, dy\,,\end{equation} where $G(x,y)$
is Green's function in $D$.  Another way to say this is that, since $\Delta G(x,\cdot)=-\delta_x(\cdot)$, integration by parts
gives $\int_D G(x,y)\,\density_2 (y)\,dy =
-\Delta^{-1}\density_2(x)$, and we obtain
\eqref{e.greencovariance} by multiplying each side by
$\density_1(x)$ and integrating with respect to $x$.

We next observe that every $h \in H(D)$
is naturally a distribution, since we may define the map $(h, \cdot)$ by $(h, \density) :=
(h, -\Delta^{-1} \density)_\nabla$.  (It is not hard
to see that $-\Delta^{-1} \density \in H(D)$, since its Dirichlet
energy is given explicitly by \eqref{e.greencovariance}; see \cite{\GFFSurvey} for more details.)

An instance of the GFF with zero boundary conditions on $D$ is a
random sum of the form $h = \sum_{j = 1}^{\infty} \alpha_j f_j$
where the $\alpha_j$ are i.i.d.\ one-dimensional standard (unit
variance, zero mean) real Gaussians and the $f_j$ are an orthonormal
basis for $H(D)$.  This sum almost surely does not converge
within $H(D)$ (since $\sum |\alpha_j|^2$ is a.s.\ infinite).  However, it does converge almost surely within the space
of distributions --- that is, the limit $(\sum_{i=1}^{\infty} \alpha_j f_j, \density)$
almost surely exists for all $\density \in H_s(D)$, and the limiting value as a function
of $\density$ is almost surely a continuous functional on $H_s(D)$ \cite{\GFFSurvey}.
We view $h$ as a sample from the measure space
$(\Omega, \mathcal F)$ where $\Omega = \Omega_D$ is the set of distributions
on $D$ and $\mathcal F$ is the smallest $\sigma$-algebra that makes
$(h, \density)$ measurable for each $\density \in H_s(D)$, and we denote by $\mu = \mu_D$
the probability measure which is the law of $h$.  The following is a standard and straightforward
result about Gaussian processes \cite{\GFFSurvey}:

\begin{proposition} \label{p.GFFdef}
The GFF with zero boundary conditions on $D$ is the only random distribution
$h$ on $D$ with the property that for each
$\density \in H_s(D)$ the random variable $(h,\density) = (h, -\Delta^{-1} \density)_\nabla$
is a mean zero Gaussian with variance
$(-\Delta^{-1} \density, -\Delta^{-1} \density)_\nabla = (\density, -\Delta^{-1} \density)$.
In particular, the law of $h$ is independent of the choice of basis $\{f_j\}$ in
the definition above.
\end{proposition}

Given a harmonically integrable function
$h_\partial : \partial D \to \R$, the {\bf GFF
with boundary conditions $h_\partial$} is the random distribution whose law is
that of the GFF with zero boundary conditions plus the deterministic function
$\underline h_\partial$ which is the harmonic interpolation of $h_\partial$ to $D$
(viewing $\underline h_\partial$ as a distribution).

Suppose $g \in H_s(D)$ and write $\density :=
-\Delta g$.  If $h$ is a GFF with boundary conditions $h_\p$ (and
$\underline h_\p$ is the harmonic extension of $h_\p$ to $D$), then
we may integrate by parts and
write $(h,\density) = (\underline h_\p, \density) + (h, g)_{\nabla}$.  Note
that $(\underline h_\p,\density)$ is deterministic
while the latter term, $(h,g)_{\nabla}$ does not depend
on $h_\p$.
 Then the random variables $(h,\density)$, for
$\density \in \Delta H_s(D)$,
have means $(\underline h_\p, \density)$ and covariances given by
$$\text{Cov} \left((h,\density_1), (h,\density_2)\right) =
(\density_1,-\Delta^{-1} \density_2) = \int_{D \times D} \density_1(x)\,
\density_2(y)\, G(x,y) \, dx \, dy\,,$$ where $G(x,y)$
is Green's function in $D$.

Green's function in the upper half plane is
$G(x,y)=(2\,\pi)^{-1}\,\log \Bigl| \frac{\overline{x}-y}{x-{y}} \Bigr|$
(where $\overline{x}$ is the complex conjugate of $x$),
since the right hand side is zero when $y\in\R$
and is the sum of $-(2\,\pi)^{-1}\log|x-y|$ and a function harmonic in $y$.
(Note that $G(x,y)>0$ for all $x,y \in \H$.)  In physical
Coulomb gas models, when $\density$ represents the density
function of a (signed) electrostatic charge distribution, the
quantity $(\density, -\Delta^{-1}\density) = \int_{D \times D}
\density(x)\, \density(y)\, G(x,y)\, dx\,dy$ is sometimes called
the {\bf electrostatic potential energy} or {\bf energy of
assembly} of $\density$ (assuming the system is grounded at the
boundary of $D$).

If $\phi$ is a conformal map from $D$ to a domain $\tilde D$ and $h$ is a distribution
on $D$, then we define the pullback $h \circ \phi^{-1}$ of $h$ to be the distribution $\tilde h$
on $\tilde D$ for which
$(\tilde h, \tilde \density) = (h, \density)$ whenever $\density \in H_s(D)$ and
$\tilde \density = |\phi'|^{-2} \density \circ \phi^{-1}$,
where $\phi'$ is the complex derivative of $\phi$ (viewing the latter as an
analytic function on a subset of $\C$).

Note that if $\density$ is interpreted in the physical sense as an electrostatic charge density on $D$
and $\phi$ is a change-of-coordinates map, then $\tilde \density$ is the corresponding charge density
on $\tilde D$.
The reader may also observe that if $h$ is a function in $H(D)$, interpreted as a distribution as
discussed above,
then the function $h \circ \phi^{-1} \in H(\tilde D)$, interpreted as a distribution, agrees
with the definition given above.

Also, if $f \in H(D)$ and $\density = -\Delta f \in H_s(D)$, then we have the following:
$$(h, f)_\nabla = (h, \density) = (h \circ \phi^{-1}, |\phi'|^{-2} \density \circ \phi^{-1}) =
(h \circ \phi^{-1}, f \circ \phi^{-1})_\nabla,$$
where the last equality follows from the simple calculus fact that
$$-\Delta (f \circ \phi^{-1}) = |\phi'|^{-2} \density \circ \phi^{-1}.$$
This and Proposition \ref{p.GFFdef} imply the conformal invariance of the GFF: i.e.,
if $h$ has the law of a GFF on $D$, then $h \circ \phi^{-1}$
has the law of a GFF on $\tilde D$.

In addition to defining an instance $h$ of the GFF with zero boundary conditions
as a random distribution, it is also occasionally useful to
define the random variables $(h, g)_\nabla$ for all $g \in H(D)$ (and not merely $g \in H_s(D)$).
Suppose that an orthonormal basis $\{f_j\}$ for $H(D)$, made up
of elements in $H_s(D)$, is fixed and we write $\alpha_j = (h, f_j)_\nabla$.
Then the $\alpha_j$ are i.i.d.\ Gaussians and for each fixed $g\in H(D)$ the sum $$(h, g)_\nabla :=
\sum_{j=1}^\infty\alpha_j\,\bigl( f_j, g\bigr)_\nabla$$ converges
almost surely.  The collection of Gaussian random variables $(h,g)_\nabla$, with
$g \in H(D)$ (defined as limits of partial sums) is a Hilbert space under the covariance inner
product.  Replacing $\{ f_j\}$ with another orthonormal basis for $H(D)$ only changes
the random variable $(h, g)_\nabla$ on a set of measure zero \cite{\GFFSurvey}.
Thus, for each fixed $g \in H(D)$, it makes sense to think of the
random variable $(h,g)_\nabla$ as a real-valued function of $\Omega$
(which is canonically defined up to redefinition on a set of measure zero).
This Hilbert space of random variables is a
closed subspace of $L^2(\Omega, \mathcal F, \mu)$.  The covariance
of $(h,g_1)_\nabla$ and $(h,g_2)_\nabla$ is equal to
$(g_1,g_2)_\nabla$, so this Hilbert space is naturally
isomorphic to $H(D)$ \cite{\GFFSurvey,\Janson}.
(In general, a collection of centered Gaussian random variables
which is a Hilbert space under the covariance inner product is
called a {\bf Gaussian Hilbert space} \cite{\Janson}.)

\subsection{GFF approximations and discrete level lines} \label{ss.GFFapprox}

In order to construct contour lines of an instance $h$ of the Gaussian free field,
it will be useful to approximate $h$
by a continuous function which is linear on each triangle in a triangular
grid.  We now describe a special case of the approximation scheme
used in \cite{\DGFFpaper}.  (The results in \cite{\DGFFpaper} also apply to more
general periodic triangulations of $\R^2$.)  Let $\TG$ be the triangular grid in the plane $\mathbb C \cong
\mathbb R^2$, i.e., the graph whose vertex set is the integer span
of $1$ and $e^{\pi i/3}=(1+\sqrt{3}\,i)/2$, with straight edges
joining $v$ and $w$ whenever $|v-w|=1$.  A {\bf $\TG$-domain} $D
\subset \mathbb R^2$ is a domain whose boundary is a simple closed
curve comprised of edges and vertices in $\TG$.

Fix some $\lambda>0$.
Let $D$ be any $\TG$-domain, and let $\p_+\subset\p D$ be an arc
whose endpoints are distinct midpoints of edges of $\TG$.  Let $V$ denote
the vertices of $\TG$ in $\closure D$. Set $\p_-:=\p D\setminus
\p_+$. Let $h_{\partial}:V \cap \partial D \to \R$ take the value
$-\,\lambda$ on $\p_-\cap V$ and $\lambda$ on $\p_+\cap V$. Let
$\phi_D$ be any conformal map from $D$ to
the upper half-plane $\H$ that maps $\p_+$ onto
the positive real ray.

Let $h^0$ be an instance of the GFF on $\H$
with zero boundary conditions.
By conformal invariance of the
Dirichlet inner product, $h^0 \circ \phi_D$ has the law of a GFF
on $D$.
  Let $H_{\TG}(D)$ be the subspace of $H(D)$ comprised of
continuous functions that are affine on each $\TG$ triangle in
$D$, and let $h^0_D$ be the orthogonal projection of $h^0$
onto $H_{\TG}(D)$
with respect to the inner product $(\cdot,\cdot)_\nabla$.
That is, $h^0_D$ is the random element of $H_{\TG(D)}$ for
which $(h^0,\cdot)_\nabla$ and $(h^0_D,\cdot)_\nabla$ are equivalent
as linear functionals on $H_{\TG(D)}$.  (The former is defined
almost surely from the Gaussian Hilbert space perspective,
as discussed above; the given definition of this projection depends on
the choice of basis $\{f_j\}$, but changing the basis affects
the definition only on a set of measure zero.)

We may view $h^0_D$ as a function on the vertices of $\TG \cap
\overline D$, linearly interpolated on each triangle
\cite{\DGFFpaper}.  Up to multiplicative constant, the law of
$h^0_D$ is that of the discrete Gaussian free field on the graph
$\TG \cap \overline D$ with zero boundary conditions on the
vertices in $\partial D$ \cite{\DGFFpaper}.  Let $h_D$ be $h^0_D$
plus the (linearly interpolated) discrete harmonic (w.r.t\ the usual
discrete Laplacian) extension of
$h_{\partial}$ to $V\cap D$.
Thus, $h_D$ is precisely the discrete field for which~\cite{\DGFFpaper}
proves that the chordal interface
of $h_D\circ\phi_D^{-1}$ converges to \SLEkk4/ for an appropriate
choice of $\lambda$.

Now $h(z) = h^0(z) + \lambda\left( 1- 2\pi^{-1}\arg
(z)\right)$ is an instance of the GFF on $\H$ with boundary
conditions $-\lambda$ on the negative reals and $\lambda$ on the
positive reals, as $\lambda\left( 1- 2\pi^{-1}\arg
(z)\right)$ is the harmonic function with this boundary data.
The functions $h_D \circ \phi_D^{-1}$
may be viewed as approximations to $h$.

There is almost surely a zero contour line $\gamma_D$ (with
no fixed parametrization) of $h_D$ on $D$ that connects the
endpoints of $\p_-$ and $\p_+$, and $\hg_D:=\phi_D \circ \gamma_D$ is a
random path in $\H$ connecting $0$ to $\infty$.  We are interested
in the limit of the paths $\hg_D$ as $D$ gets
larger.  The correct sense of \lq\lq large\rq\rq\ is measured by
$$\rr := \inr{\phi_D^{-1}(i)}(D)\,,$$ where $\inr{x}(D)$
denotes the radius of $D$ viewed from $x$, i.e., $\inf_{y \not \in
D}|x-y|$. Of course, if $\phi_D^{-1}(i)$ is at bounded distance
from $\p D$, then the image of the triangular grid under $\phi_D$
is not fine near $i$, and there is no hope for approximating \SLEkk4/
by $\hg_D$.

We have chosen to use $\H$ as our canonical domain (mapping all
other paths into $\H$), because it is the most convenient domain
in which to define chordal \SLE/.  When describing distance
between paths, it is often more natural to use the disc or another
bounded domain.  To get the best of both worlds, we will endow
$\H$ with the metric it inherits from its conformal map onto the
disc $\U$. Namely, we let $d_*(\cdot,\cdot)$ be the metric on
$\closure\H\cup\{\infty\}$ given by $d_*(z,w)=|\Psi(z)-\Psi(w)|$,
where $\Psi(z):=(z-i)/(z+i)$ maps $\closure\H\cup\{\infty\}$ onto
$\closure\U$. (Here $\closure \H$ denotes the Euclidean closure of $\H$.)  If $z\in\closure\H$, then $d_*(z_n,z)\to 0$ is
equivalent to $|z_n-z|\to 0$, and $d_*(z_n,\infty)\to 0$ is
equivalent to $|z_n|\to\infty$.



Let $\Lambda$ be the set of continuous
functions $W:[0,\infty)\to\R$ which satisfy $W_0=0$.
We endow $\Lambda$ with the topology of uniform
convergence on compact intervals---i.e., the topology generated by the sets
of the form $\{W: |W_t - V_t| < \epsilon \text{ when } 0 \leq t
\leq T \}$ for some $V_t \in \Lambda$ and $\epsilon > 0$ and $T > 0$.  Let $\mathcal
L$ be the corresponding Borel $\sigma$-algebra.

For each $W \in \Lambda$, we may define the Loewner maps
$g_t$ from subsets of $\H$ to $\H$ via the ODE
\begin{equation}\label{e.chordal} \p_t g_t(z) =
\frac {2}{g_t(z)-W_t}\,,\qquad g_0(z)=z\ . \end{equation}

When $g_t^{-1}$ extends continuously to $W_t$, we write $\gamma(t) :=
g_t^{-1} (W_t)$.  Generally, we write $\tau(z)$ to be the largest $t$ for which
the ODE describing $g_t(z)$ is defined, and $K_t = \{z\in\closure\H : \tau(z) \leq t \}$.
So $K_t$ is a closed subset of $\closure\H$, and $g_t$ is a conformal map from $\H \setminus
K_t$ to $\H$.

Let $\Lambda_C \subset \Lambda$ be the set of $W$ for which
$\gamma: [0,\infty) \to \overline \H$ is well defined for all $t \in [0, \infty)$
and is a continuous path.  The metric on paths given by
$\dstrong(\gamma_1, \gamma_2) = \sup_{t \geq 0} d_*(\gamma_1(t), \gamma_2(t))$
induces a corresponding metric on $\Lambda_C$.



\subsection{Main results} \label{ss.mainresults}

The first of our two main
results constructs a ``zero contour line of $h$'' as the limit
of the contour lines of the approximations described above:

\begin{theorem}\label{continuumcontourtheorem}
Suppose that $\lambda=\sqrt{\pi/8}$.
As $\rr\to\infty$, the random paths $\hg_D=\phi_D\circ\gamma_D$
described above, viewed as $\Lambda_C$-valued random variables on
$(\Omega, \mathcal F)$, converge in probability (with respect to
the metric $\dstrong$ on $\Lambda_C$) to an $(\Omega, \mathcal
F)$-measurable random path $\gamma \in \Lambda_C$ that is
distributed like \SLEkk 4/. In other words, for every $\eps>0$
there is some $R=R(\eps)$ such that if $\rr>R$, then
$$
\PB{\dstrong\bigl(\gamma(t),\hg_D(t)\bigr)>\eps}<\eps\,.
$$
\end{theorem}

The theorem shows, in particular, that the random path
$\gamma$ is a.s.\ determined by $h$. It will be
called the zero contour line of $h$.
We will actually initially construct $\gamma$ in
Section \ref{machinerysection} in a way that does not
involve discrete approximations. This construction is rather interesting.
In some sense, we reverse the procedure. Rather than constructing
$\gamma$ from $h$, we start with a chordal \SLEkk4/ path $\gamma$ and
define $\tilde h$ as the field that is the GFF with
boundary values $\lambda$ in the connected component
of $\H\setminus\gamma$ having $\R_+$ on its
boundary plus a conditionally independent (given $\gamma$)
GFF with boundary values $-\lambda$ in the other connected
component of $\H\setminus\gamma$.
We then show that $\tilde h$ has the same law as $h$,
effectively giving a coupling of $h$ and $\gamma$.
Much later, in Section \ref{proofsofmainresultssection}, we will
show that in this coupling $\gamma$ is
a.s.\ determined by $h$; that is, it is equal a.s.\ to
a certain $\mathcal F$-measurable $\Lambda_C$-valued function of $h$.
We will then prove Theorem \ref{uniquecontinuumpath} (stated below),
which characterizes the zero contour line of $h$ in terms of conditional expectations.

In order to make sense of the theorem stated below, we will need the fact that
the GFF on a subdomain of $\H$
makes sense as a random distribution on $\H$---and not just a random distribution
on the subdomain. We will see in Section \ref{p.subdomainGFFdef} that the GFF on a simply
connected subdomain of $\H$ is indeed canonically defined as a random distribution on $\H$.

Consider a random variable $(\tilde h,\tilde W)$
in $(\Omega \times \Lambda, \mathcal F \times \mathcal L)$.
Let $\tilde g_t:\H\setminus\tilde K_t\to\H$ denote the Loewner evolution corresponding to
$\tilde W$.
Write \begin{equation} \label{e.ht} \tilde h_t(z):= \lambda \,\bigl( 1- 2\,\pi^{-1}\arg (\tilde g_t(z)-W_t)\bigr) .\end{equation}
(This is the harmonic function on $\H\setminus\tilde K_t$ with boundary
values $-\lambda$ on the left side of the tip of the Loewner evolution
and $+\lambda$ on the right side.)

\begin{theorem} \label{uniquecontinuumpath}
Let $D = \H$ and suppose that for some $\lambda>0$ a random variable $(\tilde h,\tilde W)$
in $(\Omega \times \Lambda, \mathcal F \times \mathcal L)$ satisfies the following conformal Markov property.
For every fixed $T\in[0,\infty)$ we
have that given the restriction of $\tilde W_t$ to $[0,T]$, a regular
conditional law for $\tilde h$ on $H_s(\H\setminus\tilde K_T)$
is given by $\tilde h_T$, as in (\ref{e.ht}), plus a zero boundary GFF on $\H\setminus\tilde K_T$.
Then the following hold
\begin{enumerate}
\item $\lambda = \sqrt{\pi/8}$.
\item The trace of the Loewner evolution is almost surely a path $\tilde\gamma$
with the law of an \SLEkk 4/ (i.e., $\tilde W$ is $2$ times a standard Brownian motion).
\item \label{i.Wdetermined} Conditioned on $\tilde h$, the function $\tilde W$ is
almost surely completely determined (i.e., there
exists a $\Lambda$-valued function on $\Omega$ such that
$\tilde W$ is almost surely equal to the value of that function applied to $\tilde h$).
\item \label{i.samepathasincontinuumcontoutheorem} The pair $(\tilde h,\tilde\gamma)$ has
the same law as the pair $(h,\gamma)$ from Theorem~\ref{continuumcontourtheorem}.
\end{enumerate}
\end{theorem}

Another derivation of the \SLEk/-GFF coupling in Theorem \ref{uniquecontinuumpath} appears in \cite{MR2525778}, which references our work in progress and also explores relationships between this coupling and continuum partition functions, Laplacian determinants, and the Polyakov-Alvarez conformal anomaly formula.  The couplings in \cite{S05,MR2525778} are in fact more general than those given here; they are defined for general $\kappa$, and are characterized by conformal Markov properties similar to those in Theorem \ref{uniquecontinuumpath}.  The \SLEk/ curves in these couplings are local sets (in the sense of Section \ref{s.local}) and are interpreted in \cite{S05} as ``flow lines of $e^{ih}$'' where $h$ is a multiple of the GFF and $e^{ih}$ is viewed as a complex unit vector field (which is well defined when $h$ is smooth, and in a certain sense definable even when $h$ is a multiple of the GFF).  Indeed, these examples were the primary motivation for our definition of local sets.

\section{Coupling SLE and the GFF} \label{machinerysection}

\subsection{GFF on subdomains: projections and restrictions} \label{s.subdomain}


\begin{proposition} \label{p.projectionisadistribution}
Suppose that $H'(D)$ is a closed subspace of $H(D)$, that $P$
is the orthogonal projection onto that subspace, and that $\{f_j\}$
is an orthonormal basis for $H(D)$.  Let $\alpha_j$ be i.i.d.\
mean zero, unit variance normal random variables.  Then the sum $\sum P(\alpha_j f_j)$
converges almost surely in the space of distributions on $D$.  The law of the limit
is the unique law of a random distribution $h$ on $D$ with the property that for each $\density \in H_s(D)$,
the random variable $(h, \density)$ is a centered Gaussian with variance
$||P (-\Delta^{-1} \density)||^2_\nabla$.
\end{proposition}

\proof For each fixed $\density \in H_s(D)$, the fact that
the partial sums of
\begin{multline*}
  \Bigl(\sum P(\alpha_j f_j), \density\Bigr)
= \Bigl(\sum P(\alpha_j f_j), - \Delta^{-1} \density\Bigr)_\nabla
\\
=
\Bigl(\sum \alpha_j f_j, P(-\Delta^{-1}\density) \Bigr)_\nabla
=
\sum \alpha_j \Bigl(f_j, P(-\Delta^{-1}\density) \Bigr)_\nabla
\end{multline*}
converge to a random variable with variance $||P (-\Delta^{-1} \density)||^2_\nabla$ is
immediate from the fact that $P(-\Delta^{-1}\density) \in H(D)$.

To complete the proof, we recall the approach of \cite{\GFFSurvey} (which uses earlier
work in \cite{\Gross}) for showing
that $\sum \alpha_j f_j$ converges in the space of distributions.  Consider
a norm $\| \cdot \|_*$ of the form $\|f\|_*^2 = (Tf, f)_\nabla$ where $T$ is
a {\bf Hilbert-Schmidt operator} mapping $H(D)$ to itself---i.e., for each basis
$\{f_j\}$ of $H(D)$ we have $\sum \|T f_j\|_\nabla^2 < \infty$.
(The sum is independent of the choice of basis.)  Then Gross's classical
{\bf abstract Wiener space} result implies that the closure of $H(D)$
with respect to the norm $\|\cdot \|_*$ is a space in which $\sum \alpha_j f_j$
converges almost surely \cite{\Gross}.  In \cite{\GFFSurvey} a Hilbert-Schmidt operator $T$ was given
such that this closure was contained in the space of distributions (and such that
convergence in the norm implied weak convergence in the space of distributions).
For our purposes, it suffices to observe that if $T$ is a Hilbert-Schmidt operator
on $H(D)$ then $T' = P T$ is a Hilbert-Schmidt operator and $\|f \|_*' = (T'f,f)_\nabla$
is a measurable norm on $H'(D)$.  The existence and uniqueness of the random distribution described
in the proposition statement then follows by the standard abstract Wiener space
construction in \cite{\Gross} as described in \cite{\GFFSurvey}.
\qed

For any deterministic open subset $B$ of $D$, we denote by
$\Supp_B$ the closure in $H(D)$ of the
space of $C^\infty$ functions compactly supported on $B$ and we denote
by $\Harm_B$ the orthogonal complement of $\Supp_B$, so that
$$H(D) = \Supp_B \oplus \Harm_B.$$
Note that since $(f,g)_\nabla = (-\Delta f, g)$, a smooth function $f \in H(D)$
satisfies $f \in \Harm_B$ if and only if $-\Delta f=0$ (so
that $f$ is harmonic) on $B$.  In general, if $f$ is any distribution on $D$,
we say that $f$ is {\bf harmonic} on $B$ if $(-\Delta f, g) := (f, -\Delta g) = 0$
for all $g \in \Supp_B$.  By this definition $\Harm_B$ consists of
precisely of those elements of $H(D)$ which are harmonic on $B$.  If $f$ is a distribution
which is harmonic on $B$, then the restriction of $f$ to $B$ may be
viewed as an actual harmonic function---i.e., there is a harmonic function $\tilde f$ on $B$
such that for all $\density \in H_s(B)$ we have $(f,\density) = (\tilde f, \density)$.
We may construct $\tilde f$ explicitly as follows.  For each $z \in B$, write
$\tilde f(z) = (f, \density)$
where $\density$ is any positive radially symmetric bump function centered at $z$ whose integral
is one and whose support lies in $B$.  By harmonicity, this value is the same for any $\density$
with these properties, since the difference of two such $\density$
is the Laplacian of a (radially symmetric about $z$) function in $H_s(B)$.
The fact that this function $\tilde f$ is harmonic is easy to verify from its definition.

Suppose that $B$ is a subdomain of a simply connected planar domain $D$.  It will often be useful
for us to interpret the GFF with zero boundary conditions on $B$ as a random distribution on $D$.
This means that we have to make sense of $(h, \density)$ when $\density \in H_s(D)$ but we
do not necessarily have $\density \in H_s(B)$.
The following is immediate from Proposition \ref{p.projectionisadistribution}, taking $H'(D)$
to be $\Supp_B$.

\begin{proposition} \label{p.subdomainGFFdef}
Let $B$ be a subdomain of a planar domain $D$.  There is a unique law for a random distribution
$h_B$ on $D$ with the property that for each
$\density \in H_s(D)$ the random variable $(h_B,\density)$
is a mean zero Gaussian with variance
$$ \int_{B \times B} \density(x)\, \density(y)\, G_{B}(x,y) \, dx \, dy.$$
\end{proposition}

If we restrict the pairing $(h, \cdot)$ of the above proposition to functions
on $H_s(B)$, then by definition $h$ is the GFF on $B$ with zero boundary conditions.
The projection of the GFF onto $\Harm_B$ is a random distribution which
is almost surely harmonic on $B$.  Applying Proposition \ref{p.projectionisadistribution}
to $\Harm_B$ gives the following straightforward analog of Proposition \ref{p.subdomainGFFdef}.  (Here $\Delta_B$
denotes the Laplacian restricted to $B$, and $\Delta_B^{-1} \density$ has Laplacian $\density$ on $B$ and zero
boundary conditions on $\partial B$.)

\begin{proposition} \label{p.harmdist}
Let $B$ be a subdomain of a planar domain $D$.  There is a unique random distribution
$h^*_B$ on $D$ with the property that for each
$\density \in H_s(D)$ the random variable $(h^*_B,\density)$
is a mean zero Gaussian with variance
\begin{equation} (\density, -\Delta^{-1} \density)
-(\density, -\Delta_{B}^{-1} \density)_{B}.
\end{equation}
An instance of the GFF on $D$ may be written as $h = h_B + h^*_B$ where $h_B$
is the zero boundary GFF on $B$ and $h^*_B$ and $h_B$ are independent.
\end{proposition}

Although the above defines $h^*_B$ as a random distribution and not as a function, we may, as
discussed above, consider the restriction of $h^*_B$ to $B$
as a harmonic function.  This function intuitively represents
the conditional expectation of $h$ at points of $B$ {\em given} the values of $h$ off of $B$.

\subsection{Constructing the coupling} \label{heightevolutionsection}

The overall goal of the paper is to recognize \SLEkk4/ as an
interface determined by the GFF.  In this subsection we show how
to start with an \SLEkk4/ and use it to explicitly construct an instance
of the GFF such that the resulting coupling of \SLEkk4/ and the GFF
satisfies the hypothesis of Theorem \ref{uniquecontinuumpath}.

For any Loewner evolution $W_t$ on the half plane $\H$ we may write
$f_t(z) := g_t(z) - W_t$. Since $\arg z$ is a harmonic
function on $\H$ with boundary values $0$ on $(0,\infty)$ and
$\pi$ on $(-\infty, 0)$, the value
$\pi^{-1}\arg f_t(z)$ is the probability that a two dimensional
Brownian motion starting at $z$ first exits
$\H\setminus\gamma[0,t]$ either in $(-\infty,0)$ or on the left
hand side of $\gamma[0,t]$. In other words, for fixed $t$, the
function $\pi^{-1} \arg f_t(z)$ is the bounded harmonic function
on $\H\setminus\gamma[0,t]$ with boundary values given by $1$ on
the left side of the tip $\gamma(t)$ and $0$ on the right side.

\begin{lemma} \label{l.onepointheightevolution}
Let $B_t$ be a standard $1$-dimensional Brownian motion,
and let $W_t=2\,B_t$ be the driving parameter of an \SLEkk 4/
evolution $g_t$. Set $f_t(x):=g_t(z)-W_t$
and $h_t:= \lambda \left( 1- 2\,\pi^{-1}\arg (f_t)\right)$.  Then
for each fixed $z \in \H$, \begin{eqnarray} \label{onepointheightevolution} dh_t(z) =
\frac{4 \lambda}{\pi} \,\Im (f_t(z)^{-1})\,dB_t, \end{eqnarray}
(in the sense of It\^o differentials)
and $h_t(z)$ is a martingale.
\end{lemma}

The function $\frac{4 \lambda}{\pi}\Im (f_t(z)^{-1})$ is
significant. At time $t=0$, it is a negative harmonic function
whose level sets are circles in $\mathbb H$ that are tangent to
$\mathbb R$ at $W_0=0$. Intuitively, it represents the harmonic
measure (times a negative constant) of the tip of $K_t$ as
seen from the point $z$.  When $W_t$ moves an infinitesimal amount
to the left or right, $h_t$ changes by an infinitesimal
multiple of this function.

\proof
If $dW_t =
\sqrt{\kappa}\, dB_t$, then the \Ito/ derivatives of $f_t$ and $\log
f_t$ are as follows:
\begin{align*}
d f_t &= \frac{2}{f_t}\,dt - dW_t\,,
\\
d \log f_t &= 2\,f_t^{-2}\,dt - f_t^{-1}\,dW_t -
\frac{\kappa}{2}\,f_t^{-2}\,dt = \frac{(4-k)}{2f_t^2}\,dt -
f_t^{-1}\,dW_t\,.
\end{align*}
It is a special feature of $\kappa = 4$ that $d\log f_t =
-2f_t^{-1}\,dB_t$; hence $h_t(z)$
is a local martingale and
(since it is bounded) a martingale.
\qed

In what follows we use bracket notation $\langle X_t, Y_t\rangle := \langle X,
Y \rangle_t$ to denote the cross-variation product of
time-varying processes $X$ and $Y$ up to time $t$, i.e., $$\langle
X,Y \rangle_t := \lim \sum_{i=1}^k (X_{s_i} -
X_{s_{i-1}})(Y_{s_i} - Y_{s_{i-1}}),$$ where the limit is taken over
increasingly dense finite sequences $s_0=0< s_1 < \ldots < s_k =
t$.
Sometimes, we find it convenient to write $\langle X_t,Y_t\rangle$ in place
of $\langle X,Y\rangle_t$.
 In particular, we have almost surely $\langle B_t, B_t \rangle = t$ and
$\langle B_t, t \rangle = \langle t, B_t \rangle = \langle t,t
\rangle=0$ for all $t\geq 0$.

\begin{lemma}\label{l.dG} Suppose that $W_t$ is the driving parameter of an \SLEkk 4/
and $$h_t:= \lambda \left( 1- 2\pi^{-1}\arg (f_t)\right)$$ (as in Lemma \ref{l.onepointheightevolution}).
Let $G(x,y) = (2 \pi)^{-1}\log \left| \frac{x-\overline y}{x-y}
\right|$ be Green's function in the upper half plane and write
$G_t(x,y) = G(f_t(x), f_t(y))$ when $x$ and $y$ are both in
$\H \setminus K_t$.  If $x, y \in \H$ are fixed and $\lambda = \sqrt{\pi/8}$, then
the following holds for all $t$ almost surely:
\begin{eqnarray} \label{greensderivative} dG_t(x,y) &= & - d\langle h_t(x), h_t(y) \rangle. \end{eqnarray}
\end{lemma}

Recall that for $\kappa=4$ we have $x\notin\bigcup_{t>0}K_t$ a.s.\ for every
$x\in\H$ \cite{\RSsle}. Therefore, $G_t(x,y)$ is a.s.\ well defined for all $t$.

\proof
We first claim that
\begin{equation} \label{e.greenderivative} 2\,\pi \,d G_t(x,y) = -4\, \Im\bigl( f_t(x)^{-1}\bigr)\,
\Im\bigl( f_t(y)^{-1}\bigr)\,dt.
\end{equation}
We derive \eref{e.greenderivative} explicitly using \Ito/ calculus as follows (recalling that $g_t(x) - g_t(y) = f_t(x) = f_t(y)$):
\begin{eqnarray*} 2\,\pi \,d G_t(x,y) &=& - d\, \Re \log [g_t(x) - g_t(y)] + d\, \Re \log [g_t(x) -
\overline{g_t(y)}] \\
&=& - 2\,\Re \frac {f_t(x)^{-1} - f_t(y)^{-1}}{f_t(x) - f_t(y)}\,dt
+ {}
\\ & &\qquad 2
\,\Re \frac { f_t(x)^{-1} - \overline{f_t(y)}^{-1}}{f_t(x) - \overline{f_t(y)}}\,dt \\
&=& 2\, \Re\bigl( f_t(x)^{-1}f_t(y)^{-1}\bigr)\,dt  - 2\, \Re\bigl( f_t(x)^{-1} \bigl(\overline{f_t(y)}\bigr)^{-1}\bigr)\,dt \\ &=&  4\, \Re \bigl( i\, f_t(x)^{-1}\, \Im[ f_t(y)^{-1}]\bigr) \,dt \\
&= & -4\, \Im\bigl( f_t(x)^{-1}\bigr)\, \Im \bigl(f_t(y)^{-1}\bigr)\,dt\,.  \end{eqnarray*}

Using~\eqref{onepointheightevolution},
we have $d\langle h_t(x),h_t(y) \rangle =
\left(\frac{4 \lambda}{\pi}\right)^2 \Im\bigl( f_t(x)^{-1}\bigr)\, \Im
\bigl(f_t(y)^{-1}\bigr)\,dt$.  Setting $\lambda = \sqrt{\pi/8}$, we obtain~\eref{greensderivative}.
\qed

Next fix some $\density \in H_s(\H)$ and write $E_t(\density) := \int G_t(x,y)\, \density(x)\, \density(y)\,
dx\, dy$ for the electrostatic potential energy of $\density$ in
$\H \setminus K_t$.  For each $t$, the function $h_t$ is not
well defined on all of $\H$ (since it is not defined on $K_t$), but it is defined
(and harmonic and bounded between $-\lambda$ and $\lambda$) almost everywhere almost surely.  In
particular, when $\density \in H_s(\H)$, the integral $(h_t, \density) = \int_{\H} h_t(x)\, \density(x)\, dx$
is well defined, so we may view $h_t$ as a distribution.

\begin{lemma} \label{multipointheight}
In the setting of Lemma~\ref{l.dG}, assume $\lambda=\sqrt{\pi/8}$.
If $\density \in H_s(\H)$, then $(h_t, \density)$ is a martingale.  Moreover,
\begin{equation} \label{e.multipoint} d\langle (h_t, \density), (h_t,
\density) \rangle =- d \left( \int \density(x)\, \density(y) \,G_t(x,y)\,dx\,dy
\right)= - d E_t(\density).\end{equation} In other words, $(h_t, \density)$ is a
Brownian motion
when parameterized by minus
the electrostatic potential energy of $\density$ in $\H \setminus K_t$.  More
generally, when $\density_1, \density_2 \in H_s(D)$, we have
\begin{equation} \label{e.generalmultipoint} d\langle (h_t, \density_1), (h_t,
\density_2) \rangle =- d \left( \int \density_1(x)\, \density_2(y)\, G_t(x,y)\,dx\,dy
\right).\end{equation}
\end{lemma}

\proof  The right equality in \eref{e.multipoint} is true by definition, so it is enough to
prove \eref{e.generalmultipoint}.  Since both sides of \eref{e.generalmultipoint} are bilinear
in $\density_1$ and $\density_2$, we lose no generality in assuming that $\density_1$ and $\density_2$
are non-negative.  (Note that any $\density \in H_s(D)$ can be written $\density_1 - \density_2$ for some
non-negative $\density_1, \density_2 \in H_s(D)$; simply choose $\density_1$ to be any non-negative element of
$H_s(D)$ which dominates $\density$ and set $\density_2 = \density_1 - \density$.)

Jason Miller has pointed out in private communication a very simple way to obtain \eqref{e.generalmultipoint}.  Since $(h_t, \density_i)$ are continuous martingales, the (non-decreasing) process on the LHS of \eqref{e.generalmultipoint} is characterized by the fact that
$$(h_t, \density_1)(h_t, \density_2)- \langle (h_t, \density_1), (h_t, \density_2) \rangle$$
is a martingale.  Plugging in the RHS, we have now only to show that

\begin{equation}\label{e.intqvmartingale}
(h_t, \density_1)(h_t, \density_2) + d \left( \int \density_1(x)\, \density_2(y)\, G_t(x,y)\,dx\,dy \right)
\end{equation}
is a martingale.  By \eqref{greensderivative} $$h_t(x)h_t(y) + G_t(x,y) \,dx\,dy$$ is a martingale for fixed $x$ and $y$ in $\H$.  These martingales are uniformly bounded above for $x$ and $y$ in the support of the $\density_i$ (since $G_t(x,y)$ is non-increasing and $h_t(\cdot)$ is bounded).
Since \eqref{e.intqvmartingale} is a weighted average of these martingales, Fubini's theorem implies that \eqref{e.intqvmartingale} is itself a martingale.

Our original argument invoked a stochastic Fubini theorem (we used the one in~\cite[\S IV.4]{\ProtterBook}) and was longer and less self contained.
\qed

Now, define $h_{\infty}(z) = \lim_{t \rightarrow \infty} h_t(z)$,
$G_{\infty}(x,y) = \lim_{t \rightarrow \infty} G_t(x,y)$, and
$E_{\infty}(\density) = \lim_{t \rightarrow \infty}
E_t(\density)$. (The limit exists almost surely for fixed $z$ since $h_t(z)$ is a bounded martingale; it can also be deduced
for all $z$ from the continuity of the SLE trace.)
The reader may check that for fixed $x$,
$h_{\infty}(x)$ is almost surely
$\pm\lambda$ depending on whether $x$ is to the left or
right of $\gamma$.

Similarly,
since $G_t(x,y)$ and $E_t(\density)$ are decreasing
functions of $t$, these limits also exist almost surely.
The statement of the following lemma makes use of these
definitions and implicitly assumes Proposition \ref{p.subdomainGFFdef}
(namely, the fact that the zero boundary GFF on an arbitrary subdomain of $\H$
has a canonical definition as a random distribution on $\H$).

\begin{lemma} \label{contourplusGFF}
Assume the setting of Lemma~\ref{l.dG} and $\lambda=\sqrt{\pi/8}$.
Let $\tilde h$ be equal to $h_{\infty}$ (as defined above)
plus a sum of independent
zero-boundary GFF's, one in each component of $\H \setminus
\gamma$.  Then the law of $\tilde h$ is that of a GFF in $\H$ with
boundary conditions $-\lambda$ and $\lambda$ on the negative and
positive real axes.  In fact, the
pair $(\tilde h,W)$ constructed in this way (where $W$ is the
Loewner driving parameter of $\gamma$) satisfies the hypothesis
of Theorem \ref{uniquecontinuumpath}.
\end{lemma}
\proof
For each $\density\in H_s(D)$ the random variable $(\tilde h, \density)$
is a sum of $(\underline h_\p,\density)$, a Brownian motion started at
time zero and stopped at time $E_0(\density) -
E_{\infty}(\density)$, and a Gaussian of variance
$E_{\infty}(\density)$.  Thus its law is Gaussian with mean $(\underline h_\p, \density)$
and variance $E_0(\density)$.  The fact that
random variables $(\tilde h, \density)$ have these laws for any $\density \in H_s(\H)$ implies that $\tilde h$
is a GFF with the given boundary conditions by Proposition \ref{p.GFFdef}.  A similar
argument applies if we replace $h_0$ with $h_t$ for any stopping time $t$ of
$W_t$,
and this implies that $(\tilde h, \gamma)$ satisfies the hypothesis of Theorem \ref{uniquecontinuumpath}.
\QED

We can now prove part of Theorem \ref{uniquecontinuumpath}.

\begin{lemma} \label{lawofpathlemma}
Suppose that $(\tilde h, W)$ is a random variable whose law
is a coupling of the GFF on $\H$ (with
$\pm \lambda$ boundary conditions as above) and a real-valued
process $W = W_t$ defined for $t \geq 0$ that satisfies the
hypothesis of Theorem \ref{uniquecontinuumpath}.  Then the marginal
law of $W$ is that of $\sqrt 4$ times a Brownian motion (so that
the Loewner evolution generated by $W$ is \SLEkk 4/) and $\lambda =
\sqrt{\pi/8}$.
\end{lemma}

\proof Let $\density \in H_s(\H)$ be non-negative but not identically zero.  We first claim
that the hypothesis of Lemma \ref{lawofpathlemma} implies the conclusion of Lemma \ref{multipointheight},
namely that $(\tilde h_t, \density)$ is a Brownian motion when parameterized by $-E_t(\density)$.

Write $u = u(t) = -E_t(\density)$.  It is easy to see that $u(t)$ is continuous and strictly
increasing in $t$, at least up to the first time that $K_t$ intersects the
support of $\density$.
Write $F(\cdot)$ for the inverse of $u(\cdot)$.

Fix some constant $T>0$.  We define a process $B$ by writing
$B(u) = (\tilde h_{F(u)}, \density)$ whenever $F(u) < T$.
After time $u_0 = \sup \{u : F(u) < T \}$, we let $B$ evolve (independently of $W$)
according to the law of a standard Brownian motion until time $E_0(\density)$ (so that given
$B$ up until time $u_0$, the conditional
law of $B(u) - B(u_0)$---for each $u \geq u_0$---is a centered  Gaussian of variance $u - u_0$).  Because,
given $W$, the conditional law of $(h,\density)$ is a Gaussian with variance $E_0(\density) - u_0$,
we may couple this process $B$ with $\tilde h$ in such a way $B(E_0(\density)) = (h, \density)$ almost surely.

Now, we claim that $B$ is a standard Brownian motion on the interval $[0, E_0(\density)]$.
To see this, note that for each fixed $U>0$, the conditional law of $B(E_0(\density)) - B(U)$ (given
$B(u)$ for $u \leq U$) is that of a Gaussian of variance $E_0(\density) - U$ (independently of $B(u)$
for $u \leq U$).  It is a general fact (easily seen by taking characteristic functions)
that if $X$ and $Y$ are independent random variables, and $X$ and $X+Y$ are Gaussian, then so is $Y$.
Thus, $B(U)$ is a Gaussian of variance $U$ that is independent of $B(V)- B(U)$ for each $V>U$,
and $B(V) - B(U)$ is a Gaussian of variance $V-U$.
Since $B$ is clearly almost surely continuous, this implies that $B$ is a Brownian motion on
$[0, E_0(\density)]$.

Now, for each fixed $z \in \H$, if we take $\density$ to be
a positive, symmetric bump function centered at $z$ (with total integral one),
then the harmonicity of $\tilde h_t$ implies that $(\tilde h_t, \density) = \tilde h_t(z)$
provided that the support of $\density$ does not intersect $K_t$.  This implies
that $\tilde h_t(z)$ is a continuous martingale up until the first time that $K_t$ intersects
the support of $\density$.  Since we may take the support of $\density$ to be arbitrarily small
(and since $\tilde h_t$ is bounded between $\pm \lambda$), this implies that $\tilde h_t(z)$
is a continuous martingale
and is in fact a Brownian motion when parameterized by $E_0(\density)-E_t(\density)$.
The latter quantity satisfies $$E_0(\density) - E_t(\density) = (\density, - \Delta^{-1} \density + \Delta^{-1}_t
\density),$$
where $\Delta_t$ is the Laplacian restricted to $\H \setminus K_t$.
Since $- \Delta^{-1} \density + \Delta^{-1}_t \density$ is harmonic on $\H \setminus K_t$,
we have (provided $K_t$ does not intersect the support of $\density$),
\begin{eqnarray*} E_0(\density) - E_t(\density) &=&
\left( \Delta^{-1}\density - \Delta^{-1}_t\density, \delta_z \right) \\
&=& \left( \density, \Delta^{-1} \delta_z - \Delta^{-1}_t \delta_z \right),
\end{eqnarray*}
which is the value of the harmonic function $\Delta^{-1} \delta_z - \Delta^{-1}_t \delta_z$
at the point $z$, which is easily seen to be the log of the modulus of the derivative at $z$ of a conformal
map from $\H \setminus K_t$ to $\H$ that fixes $z$.

In order to prove the lemma, by Lemma~\ref{l.dG} it is now enough to show
that the fact that $\tilde h_t(z)$ is a Brownian motion under the time parameterization described
above determines the law of $W_t$.
At this point is is convenient to change to radial coordinates.  Let $\Psi$ be a conformal map
from $\H$ to the unit disc $\D$ sending $z$ to the origin and $\widehat K_t$ the image of
$K_t$ under $\Psi$.  Let $\widehat g_t: \D \setminus \widehat K_t \to \D$ be the conformal
map normalized to fix $0$ and have positive derivative at $0$, and let $\widehat W_t$ and $\widehat O_t$
be the arguments of the images of $W_t$ and $\infty$, respectively, under the map
$\widehat g_t\circ \Psi\circ g_t^{-1}$.  Then
$\tilde h_t(z)$ is an affine function of $\widehat W_t - \widehat O_t$,
and the time parameterization described above is
the standard radial Loewner evolution parameterization.
By Loewner's equation, $\partial_t \widehat O_t$ is a function of
$\widehat W_t-\widehat O_t$.
Therefore, the process $\widehat W_t-\widehat O_t$,
together with $\widehat W_0$ and $\widehat O_0$ determines
$\widehat O_t$ and thus also $\widehat W_t$.
This then determines $W_t$
(see \cite{math.Pr/0505368} for more details about changing between radial and chordal coordinates).
\qed

\section{Local sets}\label{s.local}

\subsection{Absolute continuity}\label{GFFfacts}

We begin this section with two simple results about singularity and
absolute continuity of the GFF.

\begin{lemma} \label{singularGFF}
Suppose that $D$ is a simply connected planar domain and that
$\underline h_\p$ is a deterministic non-identically-zero harmonic
function on $D$ and that $h$ is an instance
 of the (zero boundary) GFF on $D$.  Then $h$ and $\underline h_\p + h$
 (both of which are random distributions on $D$) have mutually singular laws.
\end{lemma}

\proof We may assume without loss of generality that $D$ is the disc
of radius $1$ centered at the origin (otherwise, we can conformally
map $D$ to this disc) and that $h_\p(0) \not = 0$. For each
$\epsilon > 0$, let $\density_\epsilon$ be a radially symmetric
positive function in $H_s(D)$ which is supported in the annulus $\{z
: 1-\epsilon < |z| < 1 \}$ and has integral one. If $h$ is an
instance of the GFF, then for each $\epsilon$, the expected value of
the Gaussian random variable $(h + h_\p, \density_\epsilon)$ is
$(h_\p, \density_\epsilon) = h_\p(0)$. It is easy to see from
\eref{e.greencovariance} that the variance of this random variable
tends to zero as $\epsilon \to 0$.  It follows from Borel-Cantelli
that for any deterministic sequence of $\epsilon$ which tend to zero
quickly enough, we will have $(h + h_\p, \density_\epsilon) \to
h_\p(0)$ almost surely and $(h , \density_\epsilon) \to 0$ almost
surely, and this implies the singularity. \qed


Say two coupled variables $X$ and $Y$ are {\bf almost independent}
if their joint law is absolutely continuous with respect to the
product of the marginal laws.  We now prove the following:

\begin{lemma}\label{aind}
Suppose that $D$ is the unit disc and that $S_1$ and $S_2$ are
connected closed subsets of $D$ such that
$$\dist(S_1,S_2):=\inf\{d(x,y):x\in S_1,\,y\in S_2\} = \epsilon>0.$$
Then the projections of the GFF on $D$ onto $\Harm_{D\setminus S_1}$
and $\Harm_{D\setminus S_2}$ are almost independent.
\end{lemma}

Informally, Lemma \ref{aind} says that the values of an instance
$h$ of the GFF on (an infinitesimal neighborhood of) $S_1$ are
almost independent of the values of $h$ on (an infinitesimal
neighborhood of) $S_2$.

\proof Since the distance between $S_1$ and $S_2$ is positive, there
exists a path $\gamma$ in $\closure D$ which is either simple or a
simple closed loop, such that the distance $\delta$ from $\gamma$ to
$S_1\cup S_2$ is positive and $\gamma$ separates $\closure{S_1}$
from $\closure{S_2}$ in $\closure D$. Let $D_1$ and $D_2$ be the
connected components of $D\setminus\gamma$ containing $S_1$ and
$S_2$.

Let $h_{\tilde D}$ be an instance of the GFF in $\tilde D = \cup
D_j$, and let $h_{\tilde D}^*$ be an independent instance of the
projection of the GFF in $D$ onto $\Harm_{\tilde D}$, as in
Proposition \ref{p.harmdist}.  Then $h = h_{\tilde D} + h_{\tilde
D}^*$ is an instance of the GFF on $D$, by Proposition
\ref{p.harmdist}.  As discussed in Section \ref{s.subdomain},
$h_{\tilde D}^*$ restricted to $\tilde D$ is a random harmonic
(though not bounded) function on $\tilde D$, and $h_{\tilde D}$ is a
sum of independent zero boundary GFFs on $D_1$ and $D_2$.

Next, we will construct continuous functions $h_1, h_2 \in H(D)$
such that each $h_i$ is equal to $h_{\tilde D}^*$ on a $\delta/3$
neighborhood of $S_i$ but vanishes on the component of $\tilde D$
not including $S_i$. For $i \in \{1,2\}$, the function $h_{\tilde
D}$ is Lipschitz on $S_i$. To see this, observe that since
$h_{\tilde D}$ is harmonic and zero on $S_i \cap \partial D$, its
gradient on $\tilde D$ extends continuously to all points on
$\partial D \setminus \gamma$.  Thus it has a maximum on each $S_i$
(since each $S_i$ is a positive distance from $\gamma$). We can then
let $h_i$ be the optimal Lipschitz extension to all of $D$ of the
function which is defined to be $h_{\tilde D}$ on the set of points
in $D$ of distance at most $\delta/3$ to $S_i$ and $0$ on $\partial
\tilde D \cup D_{3-i}$.  Since $h_i$ is Lipschitz, it must, in
particular, belong to $H(D)$.

Now, if $P_i$ denotes the projection onto the space
$\Harm_{D\setminus (S_i)}$, then we have $$(P_1(h), P_2(h)) =
(P_1(h_1) + P_1(h_{\tilde D}), P_2(h_2) + P_2(h_{\tilde D})).$$
However, $P_i(h_{\tilde D})$, for $i \in \{1,2\}$ are, independent
of one another.

The reader may easily check that if $h'$ is any projection of the
GFF onto a closed subspace of $H(D)$ and $a$ is any fixed element of
that subspace, then the law of $a + h'$ is absolutely continuous
with respect to that of $h'$.  In this case $G(h') := (h',
a)_\nabla/\|a\|_\nabla$ is almost surely well defined (once we fix a
basis of $H(D)$ comprised of members of $H_s(D)$; recall Section
\ref{ss.notation}) and is a Gaussian with zero mean and unit
variance.  The Radon-Nikodym derivative is then given by
$$\exp \bigl(-\|G(h')^2/2\| \bigr) / \exp \bigl( -(G(h') + \|a\|_\nabla)^2 \bigr).$$
Absolute continuity similarly follows if $a$ is random and
independent of $h'$. Thus the law of $$(P_1(h_1) + P_1(h_{\tilde
D}), P_2(h_2) + P_2(h_{\tilde D})).$$ is absolutely continuous with
respect to the law of the (independent, as discussed above) pair
$$(P_1(h_{\tilde D}), P_2(h_{\tilde D})),$$
which is absolutely continuous with respect to the independent
product of the marginals of $(P_1(h), P_2(h))$ by the same argument
applied to each component separately. \qed

\subsection{Local sets for discrete fields}

In this subsection only we will use the symbol $h$ to denote an
instance of the discrete GFF instead of the continuum GFF. We will
prove some basic results that will have analogs when $h$ is an
instance of the continuum GFF. If $D$ is a $\TG$-domain, then a
random subset $A$ of the set $V$ of $\TG$ vertices in $\overline D$
--- coupled with an instance $h$ of the discrete Gaussian free
field on these vertices with some boundary conditions
--- is called {\bf local} if conditioned on $A$ and the
restriction of $h$ to $A$, the law of $h$ is almost surely the
discrete GFF whose boundary conditions are the given values of $h$ on $A
\cup \partial D$.

We will now observe some simple facts about discrete local sets;
we will extend these facts in the next section to the continuum
setting.

Recall that for any deterministic subset $B$ of vertices, the
space of functions supported on $B$ is orthogonal (in the
Dirichlet inner product) to the space of functions that are
harmonic at every vertex in $B$.  Denote by $B^c$ the set of vertices
in $\overline D$ that do not lie in $B$.

\begin{lemma} \label{l.localiff} Let $h_B$ denote $h$ restricted to
a subset $B$ of the vertices of $D$.
Let $A$ be a random subset of the vertices of $D$,
coupled with an instance $h$ of the discrete Gaussian free field on $D$ with boundary
conditions $h_\p$.  Then the following are equivalent:
\begin{enumerate}
\item \label{i.Aislocal} $A$ is local.
\item \label{i.Aislocalequiv1} For each fixed subset $B\subset V\cap \closure D$, the following holds:
conditioned on $h_{B^c}$ (for almost all choices of $h_{B^c}$
in any version of the conditional probability), the event $A\cap B=\emptyset$
and the random variable $h_B$ are independent.
\item \label{i.Aislocalequiv2} For each fixed subset $B\subset V\cap \closure D$, the following holds:
let $S$ be the event that $A$ intersects $B$, and let $\tilde A$ be equal to $A$
on the event $S^c$ and $\emptyset$ otherwise.  Then conditioned on $h_{B^c}$, the
pair $(S,\tilde A)$ is independent of $h_B$.
\end{enumerate}
\end{lemma}

\proof To show that (\ref{i.Aislocal}) locality implies (\ref{i.Aislocalequiv2}), it is enough to
note that if we condition on $h_{B^c}$ and $A$, then (if $A$ does not intersect $B$),
the conditional law of $h_B$ is its conditional law given just
$h_{B^c}$. We will show this by first sampling $A$, then the values of $h$
on $A$, then $h_{B^c}$, then $h_B$.  By the locality definition, the conditional law
of $h$, given $A$ and $h_A$, is that of a DGFF whose boundary conditions
are the given values of $h$ on $A \cup \partial D$.  Therefore, after further conditioning on the event
$S$ and $h_{B^c}$, the conditional law of $h$ will be that
of the DGFF with the given heights on $A \cup \partial D \cup B^c$.  In particular,
once $h_{B^c}$ and $S^c$ are given, the conditional law of $h$ does not depend on $A$.
Since this conditional law of $h$ is
the same as the conditional law given only $h_{B^c}$ (and no information about $S$ or $A$), it follows
that given $h_{B^c}$, the conditional law of $h$ is independent of $(S, \tilde A)$.


Next, (\ref{i.Aislocalequiv2}) clearly implies (\ref{i.Aislocalequiv1}).
We will now assume (\ref{i.Aislocalequiv1}) and derive (\ref{i.Aislocal}).
To do this, we prove the statement ``If $\P[A=C]\ne 0$,
then conditioned on $A=C$ and
the heights on $C\cup \p D$, the law of $h$ is the law of a DGFF given those
heights'' by induction on the size of $C$. The statement is clearly true if $C$ is
empty.  For general $C$, we know from our assumption that conditioned on $h_C$
and $A \subset C$, the law of $h$ is the law of a DGFF with the given
heights on $C$.  By the inductive hypothesis, we know that if we condition
on $h_C$ and any particular choice for $A$ which is a proper subset of
$C$, we will have this same conditional law; considering separately the cases $A=C$ and $A$ a proper
subset of $C$, it follows that we also
have this law if we condition on $h_C$ and $A=C$, provided that $\P[A=C]\ne0$.
\qed

The inductive technique used to derive (\ref{i.Aislocal}) from (\ref{i.Aislocalequiv1}) above is a fairly
general one.  We will make reference to it later in the paper as well.

We say that $A$ is {\bf algorithmically local} if there is a
(possibly random) algorithm for generating $A$ from the values of
$h$ on vertices of $D$ such that almost surely every vertex whose
height $h(v)$ is used by the algorithm is included in $A$. It
is easy to see that algorithmically local sets are local:

\begin{proposition}
If a random set $A$ coupled with the discrete GFF on the vertices
of $D$ with boundary conditions $h_\p$ is algorithmically local,
then it is also local.
\end{proposition}

For a trivial example, $A$ is algorithmically local whenever its
law is independent of $h$---in particular, if $A$ is
deterministic. The set $\{v:h(v) < 0 \}$ is not
local; however, the set of $v$ that are in or adjacent to a
boundary-intersecting component of the subgraph of $\TG$ induced
by $\{v:h(v) < 0\}$ {\em is} algorithmically local.  Another algorithmically local set is the set of hexagons
on either side of the discrete interface in Figure \ref{f.interface}.

Given two distinct random sets $A_1$ and $A_2$ (each coupled with a
discrete GFF $h$), we can construct a three way coupling $(h, A_1, A_2)$ such that
the marginal law of $(h, A_i)$ (for $i \in \{1,2\}$) is the given
one, and conditioned on $h$, the sets $A_1$ and $A_2$ are
independent of one another.  This can be done by first sampling $h$ and
then sampling $A_1$ and $A_2$ independently from the regular conditional
probabilities.  The union of $A_1$ and $A_2$ is then a new random set
coupled with $h$.  We denote this new random set by $A_1 \ccup A_2$ and
refer to it as the {\bf conditionally independent union} of
$A_1$ and $A_2$.


\begin{lemma}\label{l.independentpair}
Suppose that $\mathcal A$ and $\mathcal B$ and $\mathcal C$ are $\sigma$-algebras
on which a probability measure is defined for which
\begin{enumerate}
\item $\mathcal A$ is independent of $\mathcal B$,
\item $\mathcal A$ is independent of $\mathcal C$, and
\item given $\mathcal A$, the $\sigma$-algebras $\mathcal B$ and $\mathcal C$ are independent of each other.
\end{enumerate}
Then $\mathcal A$ is independent of the $\sigma$-algebra generated by both $\mathcal B$ and $\mathcal C$.
\end{lemma}
\proof
Let $A,B,C$ be events in $\mathcal{A,B,C}$, respectively, then
\begin{multline*}
\Pb{A\cap B\cap C} =
\Eb{ \Pb{B\cap C\cap A \md \mathcal A}}
=\Eb{\Pb{B\cap C\md \mathcal A}\,1_A}
\\
=\Eb{
\Pb{B\md \mathcal A}\,
\Pb{C\md \mathcal A}\,
1_A}
=\Eb{ \Ps{B}\, \Ps{C}\, 1_A}
= \Ps{B}\,\Ps{C}\,\Ps{A}\,.
\qed
\end{multline*}

\begin{lemma} \label{discretelocalunion}
If $A_1$ and $A_2$ are local sets coupled with $h$, then $A:=A_1 \ccup A_2$ is also local.
In fact, we have the slightly stronger statement that given the
pair $(A_1,A_2)$ and the values of $h$ on $A$, the conditional law of $h$ off
of $A$ is that of a DGFF with the given boundary values.
\end{lemma}


\proof Let $S_1$ and $S_2$ be the events that $A_1$ and $A_2$ hit $B\subset \closure D\cap V$,
respectively.
Let $\tilde A_i$
be equal to $A_i$ on the event $S_i^c$ and $\emptyset$ otherwise.  For almost
all $h_{B^c}$ we have (by Lemma \ref{l.localiff}) that conditioned on $h_{B^c}$
\begin{enumerate}
\item $(S_1, \tilde A_1)$ is independent of $h_B$,
\item $(S_2, \tilde A_2)$ is independent of $h_B$, and
\item given $h_B,$ the events $(S_1,\tilde A_1)$ and $(S_2,\tilde A_2)$ are independent of each other.
\end{enumerate}
(The last item follows from the definition of conditionally independent
union.) Then Lemma \ref{l.independentpair} implies that conditioned on $h_{B^c}$,
the random variable $h_B$ must be independent of $(S_1, S_2, \tilde A_1, \tilde A_2)$.
In particular, this shows that $h_B$ is independent of the union of the events
$S_1$ and $S_2$, which implies that the conditionally independent union of $A_1$ and
$A_2$ is local by Lemma \ref{l.localiff}.  It also shows the final claim in Lemma \ref{discretelocalunion},
namely that conditioned on $A_1 \cup A_2 =C$ and
the values of $h$ on $C$ and on the pair $(A_1, A_2)$, the law of $h$ is the law of a DGFF given those
values.
\qed

The following is an immediate consequence of the definition
of a local set.
\begin{proposition} \label{discretelocalharmonic}
If $A$ is a local set, then conditioned on $A$ (in any regular
conditional probability) the expected value of $h(v)$ is, as a
function of $v$, almost surely harmonic in the complement of $A$.
\end{proposition}

Note that conditioned on $A$, the restriction of $h$ to $A$ is not
deterministic; thus we would not expect the expectation of $h$
conditioned on $A$ to be the same as the expectation conditioned on $A$ {\em and} the values of $h$ in $A$ (though something like this
will turn out to hold for our continuum level sets).

\begin{remark}
Although we will not use this fact here, we remark that Lemmas \ref{l.localiff} and \ref{discretelocalunion}
are true in much greater generality.  Suppose that $h$ is any random function from a finite
set $V$ to a measure space $X$, and that for each $f:V \to X$ and each subset $B$ of $V$ we are given
a probability measure $\Phi(f,B)$ on functions from $V$ to $X$, and for each $B$, the measure $\Phi(f,B)$
is a regular conditional probability for $h$ given its values on $V \setminus B$.  (In the case of the DGFF on
a graph $G$ with vertex set $V$, this $\Phi(f,B)$ is simply the DGFF with boundary conditions given by
the values of $f$ on $V \setminus B$ and the original boundary vertices.)  We can then define a random
set $A$ coupled with $h$ to be {\bf local} if $\Phi(h,A)$ is a regular version of
the conditional probability of $h$ given $A$ and the values of $h$ on $A$.  The proofs of
Lemmas \ref{discretelocalunion} and \ref{l.localiff} apply in this generality without modification.
\end{remark}

\subsection{Local sets for the GFF}


Let $\Gamma$ be the space of all closed (with respect to the $d_*$ metric) nonempty subsets of
$\overline\H \cup \{\infty\}$.
We will always view $\Gamma$ as a metric space, endowed with the
{\bf Hausdorff} metric induced by $d_*$, i.e., the distance
between sets $S_1, S_2 \in \Gamma$ is
$$\dhaus(S_1,S_2):=\max\Bigl \{ \sup_{x \in S_1} d_*(x, S_2), \sup_{y \in S_2} d_*(y,
S_1)\Bigr \}.$$  Note that $\Gamma$ is naturally equipped with the Borel $\sigma$-algebra on
$\Gamma$ induced by this metric.  It is well known (and the reader
may easily verify) that $\Gamma$ is a compact metric space.  Note that the elements of $\Gamma$ are themselves
compact in the $d_*$ metric.

Given $A \subset \Gamma$, let $A_\delta$ denote the closed set containing all points in $\H$
whose $d_*$ distance from $A$ is at most $\delta$.
Let $\mathcal A_\delta$ be the
smallest $\sigma$ algebra in which $A$ and the restriction of $h$
(as a distribution) to the interior of $A_\delta$ are measurable.
Let $\mathcal A = \bigcap_{\delta \in \mathbb Q, \delta > 0} \mathcal A_\delta$.
Intuitively, this is the smallest $\sigma$-field in which $A$ and the values of
$h$ in an infinitesimal neighborhood of $A$ are measurable.

\begin{lemma} \label{l.localdefequivalence}
Let $D$ be a simply connected planar domain, suppose that $(h, A)$ is
a random variable which is a coupling of an instance $h$ of the GFF with
a random element $A$ of $\Gamma$.  Then the following are equivalent:
\begin{enumerate}
\item \label{i.Bcond} For each deterministic open $B \subset D$, we have that {\em given} the
projection of $h$ onto $\Harm_B$, the event $A \cap B = \emptyset$
is independent of the projection of $h$ onto $\Supp_B$.  In other words,
the conditional probability that $A \cap B = \emptyset$ given $h$ is a measurable function of
the projection of $h$ onto $\Harm_B$.

\item \label{i.B2cond} For each deterministic open $B \subset D$, we have that {\em given} the
projection of $h$ onto $\Harm_B$, the pair $(S, \tilde A)$
(defined as in Lemma \ref{l.localiff}) is independent of the projection of $h$ onto $\Supp_B$.

\item \label{i.Acond} Conditioned on $\mathcal A$, (a regular
version of) the conditional law of $h$ is that of $h_1 + h_2$ where
$h_2$ is the GFF with zero boundary values on $D\setminus A$
(extended to all of $D$ via Proposition~\ref{p.subdomainGFFdef})
and $h_1$ is an $\mathcal A$-measurable random distribution
(i.e., as a distribution-valued function on
the space of distribution-set pairs $(h, A)$, $h_1$ is $\mathcal A$-measurable)
which is a.s.\ harmonic on $D \setminus A$.
\item \label{i.twostep} A sample with the law of $(h,A)$ can be produced as
follows.
  First choose the pair $(h_1, A)$
according to some law where $h_1$ is almost surely harmonic on $D \setminus A$.  Then
sample an instance $h_2$ of the GFF on $D \setminus A$ and set $h=h_1+h_2$.
\end{enumerate}
\end{lemma}

Following the discrete definitions, we say a random closed set $A$
coupled with an instance $h$ of the GFF, is {\bf local} if
one of the equivalent items in Lemma \ref{l.localdefequivalence} holds.
For any coupling of $A$ and $h$, we use the notation $\condexp_A$ to describe
the conditional expectation of the distribution $h$ given
$\mathcal A$.  When $A$ is local, $\condexp_A$ is the $h_1$ described in item (\ref{i.Acond}) above.

\proof
Trivially, (\ref{i.B2cond}) implies (\ref{i.Bcond}).
Next, suppose $A$ satisfies (\ref{i.Bcond}).  We may assume that $D$ is bounded (applying a conformal map if necessary
to make this the case).  Fix $\delta$ and let $\widehat A_\delta$ denote the intersection of $D$ with
the union of all closed squares of the grid
$\delta \mathbb Z^2$ that intersect $A_\delta$.  Then we claim that
$\widehat A_\delta$ satisfies (\ref{i.Bcond}) as well for each deterministic choice of $\delta$.
This can be seen by replacing $B$ with $B' := D \setminus (D \setminus B)_\delta$, and noting that $A$ intersects $B'$
if and only if $A_\delta$ intersects $B'$.  Since $B'\subset B$, conditioning on $\Harm_{B'}$ is equivalent to conditioning on $\Harm_{B}$
and then conditioning on a function of $\Supp_B$ (and $\Supp_{B'}$ is also a function of $\Supp_B$), which proves the claim.

There are only finitely many possible choices for $\widehat A_\delta$, so the fact
that $\widehat A_\delta$ satisfies (\ref{i.Acond}) follows by the inductive argument used in the proof
of Lemma \ref{l.localiff}.

To be precise, we prove the statement ``If $\P[\widehat A_\delta =C]\ne 0$,
then conditioned on $\widehat A_\delta =C$ and
the projection $h_1$ of $h$ onto the space of functions harmonic off of $C$,
the law of $h$ is the law of a zero-boundary DGFF on $D \setminus C$ plus $h_1$''
by induction on the size of $C$. The statement is clearly true if $C$ is
empty.  For general $C$, we know from our assumption that conditioned on $h_C$
and $A \subset C$, the law of $h$ is the law of a DGFF with the given
heights on $C$.  By the inductive hypothesis, we know that if we condition
on $h_C$ and any particular choice for $A$ which is a proper subset of
$C$, we will have this same conditional law; it follows that we also
have this law if we condition on $h_C$ and $A=C$, provided that $\P[A=C]\ne0$.

Since $\mathcal A$ is the intersection of the $\widehat {\mathcal A}_\delta$ (defined analogously
to $\mathcal A_\delta$), for
$\delta > 0$, the reverse martingale convergence theorem implies the almost sure
convergence $\condexp_{\widehat A_\delta} \to \condexp_A$ as $\delta \to 0$ in the weak sense,
i.e., for each fixed $\density$, we have a.s.\
$$(\condexp_{\widehat A_\delta}, \density) \to (\condexp_A, \density).$$
This and the fact that (\ref{i.Acond}) holds for every $\widehat A_\delta$ implies
that it must hold for $A$ as well.  Since this holds for every fixed $\density$, we may extend to all
$\density \in H_s(D)$ and obtain (\ref{i.Acond}) by Proposition~\ref{p.GFFdef} and Proposition~\ref{p.subdomainGFFdef}.

Now (\ref{i.twostep}) is immediate from (\ref{i.Acond}) when we set $h_1 = \condexp_A$.
To obtain (\ref{i.B2cond}) from (\ref{i.twostep}), if suffices to show that given the
projection of $h$ onto $\Harm_B$ {\em and} the pair $(S, \tilde A)$,
the conditional law of the projection of $h$ onto $\Supp_B$ is the same as its
a priori law (or its law conditioned on only the
projection of $h$ onto $\Harm_B$), namely the law of the zero boundary GFF on
$B$.  To see this, we may first sample $A$ and $h_1$ and then---conditioned
on $A \cap B = \emptyset$---sample the projection of $h - h_1$ onto $\Supp_B$.
Since the law of $h-h_1$ is the GFF on $D \setminus A$ by assumption, this
projection is the GFF on $B$, as desired.
\qed

\begin{lemma} \label{continuumlocalunion}
Lemma \ref{discretelocalunion} applies in the continuum setting as
well.  That is, if $A_1$ and $A_2$ are local sets coupled with
the GFF $h$ on $D$, then their conditionally independent
union $A = A_1 \ccup A_2$ is also local.  The analog of the slightly
stronger statement in Lemma \ref{discretelocalunion} also holds:
given $\mathcal A$ {\em and} the pair $(A_1, A_2)$, the conditional law of $h$
is given by $\condexp_A$ plus an instance of the GFF on $D \setminus A$.
\end{lemma}

\proof The proof is essentially identical to the discrete case.  We
use characterization (\ref{i.B2cond}) for locality as given
in Lemma \ref{l.localdefequivalence} and observe that Lemma
\ref{l.independentpair} implies the analogous result holds for
the quadruple $(S_1, \tilde A_1, S_2, \tilde A_2)$---namely, that for each deterministic open $B \subset D$, we
have that {\em given} the projection of $h$ onto $\Harm_B$ and the
quadruple $(S_1, \tilde A_1, S_2, \tilde A_2)$, 
the conditional law
of the projection of $h$ onto $\Supp_B$ (assuming $A \cap B = \emptyset$)
is the law of the GFF on $B$.
 The proof that this analog of (\ref{i.B2cond})
implies the corresponding analog of (\ref{i.Acond}) in the statement of
Lemma \ref{continuumlocalunion} is also essentially as the discrete case.
\qed

\begin{lemma} \label{doublelocal}
Let $A_1$ and $A_2$ be connected local sets.  Then $\condexp_{A_1 \ccup A_2} -
\condexp_{A_2}$ is almost surely a harmonic function in
$D \setminus (A_1 \ccup A_2)$ that
tends to zero on all sequences of points in $D \setminus (A_1 \ccup A_2)$
that tend to a limit in $A_2\setminus A_1$ (unless $A_2$
is a single point).
\end{lemma}

\proof By Lemma \ref{continuumlocalunion}, the union $A_1 \ccup A_2$ is
itself a local set, so $\condexp_{A_1 \ccup A_2}$ is well defined.
Now, conditioned on $\mathcal A_1$ the law of the field in $D
\setminus A_1$ is given by a GFF in $D \setminus A_1$ plus
$\condexp_{A_1}$.  We next claim that $\closure{A_2
\setminus A_1}$ is a local subset of $D \setminus A_1$, with
respect to this GFF on $D \setminus A_1$.
To see this, note that characterization (\ref{i.Bcond}) for locality from
Lemma \ref{l.localdefequivalence}
follows from the latter statement in Lemma \ref{continuumlocalunion}.

By replacing $D$ with $D \setminus A_1$ and subtracting $\condexp_{A_1}$,
we may thus reduce to the case that $A_1$ is deterministically empty and
$\condexp_{A_1} = 0$.  What remains to show is that if $A$ is any local
set on $D$ then $\condexp_A$ (when viewed as a harmonic function on $D \setminus A$)
tends to zero almost surely along all sequences of points in $D \setminus A$
that approach a point $x$ that lies on a connected components of $\partial D \setminus A$ that consists
of more than a single point.

If we fix a neighborhood $B_1$ of $x$ and another neighborhood $B_2$ whose distance from $B_1$ is
positive, then the fact that the statement holds on the event $A \subset B_2$ is immediately from
Lemma \ref{singularGFF} and Lemma \ref{aind}. Since this holds for arbitrary $B_1$ and $B_2$,
the result follows. \qed

To conclude, we note the following is immediate from the definition of local and
Theorem \ref{uniquecontinuumpath}.

\begin{lemma} \label{l.gammaislocal}
In the coupling between $h$ and $\gamma$ of Theorem \ref{uniquecontinuumpath},
the set $\gamma([0,T])$ is local for every $T \geq 0$.  The same is true if $T$ is
a non-deterministic stopping time of the process $W_t$.
\end{lemma}

\section{Fine grid preliminaries}
\subsection{Subspaces are asymptotically dense} \label{ss.densesubspaces}

As $\rr \rightarrow \infty$, the subspaces $\{g\circ\phi_D^{-1}:g\in H_{\TG}(D)\}$,
 where $\phi_D$ and $\rr$ are as defined in Section \ref{ss.GFFapprox},
become asymptotically
dense in $H(\H)$ in the following sense.  When $D$ is a $\TG$-domain and $g \in H_s(D)$,
let $P_D(g)$ denote the orthogonal projection of $g$ (with respect to
the inner product $(\cdot,\cdot)_\nabla$)
onto the space of continuous functions which are affine on each triangle of $\TG$.
For $f\in H(\H)$ set $f_D:=f\circ \phi_D$.


\begin{lemma}\label{projectiveconvergence}
Let $D\subset\C$ denote a $\TG$-domain, and assume the notation above. For each $f \in
H(\H)$, the values
$\bigl\|P_D(f_D)\circ\phi_D^{-1}-f\bigr\|_\nabla=\|P_D(f_D) - f_D\|_\nabla$ tend to zero as $\rr
\rightarrow \infty$.  In fact, if $f \in H_s(\H)$, then $\|P_D(f_D) -
f_D\|_\nabla = O(\frac{1}{\rr})$, where the implied constant may
depend on $f$.
\end{lemma}

\proof Since $H_s(\H)$ is dense in $H(\H)$, the former statement
follows from the latter.  Suppose that $f \in H_s(\H)$.  Then it
is supported on a compact subset $K$ of $\H$.

When $z$ ranges over values in $K$ and $\phi$ ranges over {\em
all} conformal functions that map a subdomain of $\C$ onto $\H$,
standard distortion theorems for conformal functions (e.g.,
Proposition 1.2 and Corollary 1.4 of \cite{\PommeBDRY}) imply the
following (where the implied constants may depend on $K$):

\begin{enumerate}
\item The ratio of $|(\phi^{-1})'(i)|$ and $\rr :=
\inr{\phi^{-1}(i)}(\phi^{-1}\H)$ is bounded between two
positive constants.
\item $|(\phi^{-1})'(z)| = O(|(\phi^{-1})'(i)|) = O(\rr)$.
\item $\diam (\phi^{-1}(K)) = O(|(\phi^{-1})'(i)|)=O(\rr)$.
\item $|\phi'(\phi^{-1}(z))| = O(\frac{1}{\rr})$.
\item $|\phi''(\phi^{-1}(z))| = O\bigl(\frac{1}{(\phi^{-1})'(z)^2}\bigr) = O\bigl(\rr^{-2}\bigr)$.
\end{enumerate}

Now $\|P_D(f_D) - f_D\|_\nabla = \inf\{ \|g - f_D\|_\nabla:g\in H_{\TG}(D)\}$.
We will bound the
latter by considering the case that $g$ is the
function $g_D\in H_{\TG}(D)$ that agrees with $f_D$ on
$\TG$ --- and then applying the above bounds with $\phi = \phi_D$.

Since the triangles of $\TG$ contained in $D$ have side length
one, the value $|\nabla g_D - \nabla f_D|$
on a triangle is bounded by a
constant times the maximal norm of the second derivative matrix of $f_D=f
\circ \phi_D$ in that triangle (where the latter is viewed as a function from
$\mathbb R^2$ to $\mathbb R$). If $f$ and $\phi_D$ were both
functions from $\mathbb R$ to $\mathbb R$, then the chain rule
would give $$(f \circ \phi_D)''(z) = [f'(\phi_D(z))\phi_D'(z)]' =
f''(\phi_D(z))\phi_D'(z)^2 + f'(\phi_D(z))\phi_D''(z).$$ In our
case, when we view $f$ as a function from $\mathbb R^2$ to
$\mathbb R$ and $\phi_D'$ as a function from $\mathbb R^2$ to
$\mathbb R^2$, the chain rule yields the formulas: but now $f'$ at
a point is understood to be a linear map from $\mathbb R^2$ to
$\mathbb R$, and $\phi_D'$ at a point is understood to be a linear
map from $\mathbb R^2$ to $\mathbb R^2$, etc.

Since all components of $f'$ and $f''$ are bounded on $K$, the
distortion bounds above give $$|(f \circ \phi_D)''(z)| =
O(|\phi_D'(z)^2 + \phi_D''(z)|) = O\bigl(\rr^{-2}\bigr)$$ and hence
$$\|\nabla g_D - \nabla f _D\|_\infty^2 =
O\bigl(\rr^{-4}\bigr).$$ The area of the support of $f \circ \phi_D$
is $O([\diam \phi_D^{-1}(K)]^2) = O(\rr^2)$. Thus $\|g_D - f\circ
\phi_D\|^2_\nabla = O(\rr^2/\rr^4)= O(\rr^{-2})$. \QED

\subsection{Topological and measure theoretic preliminaries}

In this section we assemble several simple topological facts that
will play a role in the proof of Theorem \ref{continuumcontourtheorem}.
Up to this point, we have treated the space $\Omega = \Omega_D$ of distributions on
a planar domain $D$ as a measure space, using $\mathcal F$ to represent
the smallest $\sigma$-algebra that makes $(\cdot, \density)$ measurable for
each fixed $\density \in H_s(D)$.  We have not yet explicitly introduced a metric
on $\Omega$.  (When we discussed convergence of distributions, we implicitly
used the weak topology---i.e., the topology in which $h_i \to h$ if and only if $(h_i, \density)
\to (h, \density)$ for all $\density \in H_s(D)$.)
Although it does not play a role in our main theorem statements, the following
lemma will be useful in the proofs.  Recall that a topological space is called {\bf $\sigma$-compact}
if it is the union of countably many compact sets:

\begin{lemma}\label{l.nicemetric}
Let $D$ be a simply connected domain.  There exists a metric $\widehat d$ on a subspace
$\widehat \Omega \subset \Omega$ with $\widehat \Omega \in \mathcal F$ such that \begin{enumerate}
\item An instance of the GFF on $D$ lies in $\widehat \Omega$ almost surely.
\item The topology induced by $\widehat d$ on $\widehat \Omega$ is $\sigma$-compact.
\item The Borel $\sigma$-algebra on $\widehat \Omega$ induced by $\widehat d$ is
the set of subsets of $\widehat \Omega$ that lie in $\mathcal F$.
\end{enumerate}
\end{lemma}


\proof It is enough to prove Lemma \ref{l.nicemetric} for a single bounded simply connected domain $D$ (say the unit
disc), since pulling back the metric $\widehat D$ via a conformal map preserves the properties
claimed in the Lemma.
When $f_i$ is an eigenvalue of the Laplacian with negative eigenvalue $\lambda$,
then we may define $(-\Delta)^a f_i = (-\lambda)^a f_i$, and we may extend this definition
linearly to the linear span of the $f_i$.
Denote by $(-\Delta)^a L^2(D)$ the Hilbert space closure of the linear span of the eigenfunctions
of the Laplacian on $D$ (that vanish on $\p D$) under the inner product $(f, g)_a := ( (-\Delta)^{-a} f,
(-\Delta)^{-a} g)$.
In other words, $(-\Delta)^a L^2(D)$ consists of those $f$ for which
$(-\Delta)^{-a} f \in L^2(D)$.  It follows immediately from Weyl's formula for bounded domains
that $(-\Delta)^a L^2(D) \subset (-\Delta)^b L^2(D)$
when $a < b$, that each of these spaces is naturally a subset of $\Omega$, and
that when $a>0$, an instance of the GFF almost surely lies in $(-\Delta)^a L^2(D)$.
(See \cite{\GFFSurvey} for details.)
We can thus take $\widehat \Omega = (-\Delta)^a L^2(D)$ for some $a > 0$
and let $\widehat d$ be the Hilbert space metric corresponding to $(-\Delta)^b L^2(D)$ for
some $b > a$.

To see that the topology induced by $\widehat d$ on $\widehat \Omega$ is $\sigma$-compact follows
from the fact that with its usual Hilbert space metric, $(-\Delta)^a L^2(D)$ is separable (in particular
can covered with countably many translates of the unit ball), and that such a unit ball is compact w.r.t.\
$\widehat d$.

We next argue that the Borel $\sigma$-algebra $\widehat {\mathcal F}$ on $\widehat \Omega$ induced by $\widehat d$ is
the set of subsets of $\widehat \Omega$ that lie in $\mathcal F$.  Recall that $\mathcal F$ is the smallest $\sigma$-algebra that makes
$(h, \density)$ measurable for each $\density \in H_s(D)$.  Clearly a unit ball of $\widehat d$ is in this $\sigma$-algebra (since it $H_s(D)$
is dense in such a unit ball), which shows that $\widehat{\mathcal F} \subset \mathcal F$.  For the other direction, it suffices to observe that each generating subset of $\mathcal F$ of the form $ \{ (\cdot, \density) \leq c \}$, with $c \in \R$ and $\density \in H_s(D)$ has an intersection with $\widehat \Omega$ that belongs to $\widehat{\mathcal F}$.
\qed

We now cite the following basic fact (see \cite[Thm.~72,73]{webtutorial} or \cite[Ch.~7]{\Dudley}):

\begin{lemma}\label{l.regular}
Every $\sigma$-compact metric space is separable.
If $\mu$ is a Borel probability measure on a $\sigma$-compact metric space
then $\mu$ is regular, i.e., for each Borel measurable set $S$, we have
$$\mu(S) = \inf \mu(S') = \sup \mu(S''),$$
where $S'$ ranges over open supersets of $S$ and $S''$ ranges over
compact subsets of $S$.
\end{lemma}

A family of probability measures $\mu$ on a separable topological space
$X$ is said to be {\em tight} if for every $\epsilon>0$ there is a compact
$X' \subset X$ such that $\mu(X') > 1-\epsilon$ for every $\mu$ in the family.
Prokhorov's theorem states that (assuming $X$ is separable) every tight family
of probability measures on $X$ is weakly pre-compact.  If $X$ is a separable metric space, then the converse holds,
i.e., every weakly pre-compact family of probability measures is also tight.

\begin{lemma} \label{l.weaklimit}
If $\Theta_1$ and $\Theta_2$ are two weakly pre-compact families of probability measures
on complete separable metric spaces $Z_1$ and $Z_2$, then the space of couplings between measures in the two families is weakly pre-compact.
\end{lemma}
\proof By the converse to Prokhorov's theorem, $\Theta_1$ and $\Theta_2$ are both tight.
This
implies that the space of couplings between elements of $\Theta_1$ and $\Theta_2$
is also tight, which in turn implies pre-compactness
(by Prokhorov's theorem).
\qed

The following is another simple topological observation that will be useful later on:
\begin{lemma} \label{l.weakgraphconvergence}
Suppose that $Z_1$ and $Z_2$ are complete separable metric spaces, $\mu$ is a Borel probability measure
on $Z_1$ and $\pphi_1, \pphi_2, \ldots$ is a sequence of measurable functions from $Z_1$
to $Z_2$.  Suppose further that when $z$ is a random variable distributed
according to $\mu$, the law of $(z, \pphi_i(z))$ converges weakly to that of $(z, \pphi(z))$ as
$i \to \infty$, where $\pphi:Z_1 \to Z_2$ is Borel measurable.  Then the functions $\pphi_i$,
viewed as random variables on the probability space $Z_1$, converge to $\pphi$ in probability.
\end{lemma}
\proof
By tightness of the set of measures in the sequence (recall Lemma \ref{l.weaklimit}),
for each $\epsilon>0$, we can find a compact $K \subset Z_2$ such that
$\mu(\pphi^{-1}(K)) > 1-\epsilon$.
Let $B_1, \ldots, B_k$ be a finite partition
of $K$ into disjoint measurable sets of diameter at most $\epsilon$ and write $C_j = \pphi^{-1} B_j$
for each $j$.
Then Lemma \ref{l.regular} implies that there exist open subsets
$C'_1, \cdots C'_k$ of $Z_1$ such that $C'_j \supset C_j$ for each $j$
and $\sum_j \mu (C'_j \setminus C_j) \leq \epsilon$.
For each $j$, let $B'_j$ be the set of points of distance at most
$\epsilon$ from $B_j$.  Let $\tilde \mu_i$ denote the law of $(z, \pphi_i(z))$ and
$\tilde \mu$ the law of $(z, \pphi(z))$.
Set $A':=\bigcup_{j=1}^k C_j'\times B_j'$ and $A:=\bigcup_{j=1}^k C_j\times B_j'$.
Then a standard consequence of weak convergence
(Portmanteau's theorem \cite[Theorem~11.1.1]{\Dudley}) implies
$ \liminf_{i\to\infty} \tilde\mu_i(A')\ge\tilde\mu(A')$.
But $$\tilde\mu(A')\ge \tilde\mu(A)= \sum_{j=1}^k\mu(C_j) > 1-\epsilon\,.$$
Hence $\liminf_{i\to\infty}\tilde\mu_i(A') >1-\epsilon$.
Since $\sum_{j=1}^k \mu(C_j'\setminus C_j)\le\epsilon$,
we have $\liminf_{i\to\infty}\tilde\mu_i(A)\ge 1-2\epsilon$.
But when $(z,\pphi_i(z))\in A$, the distance between $\pphi_i(z)$ and $\pphi(z)$ is
at most $2\,\epsilon$.
Hence,
$$\liminf_{i \to \infty} \mu \{x\in Z_1 : d(\pphi_i(x),\pphi(x)) \leq 2\epsilon \} \geq 1-2\epsilon.$$
Since this holds for any $\epsilon>0$, the result follows.
\qed

\subsection{Limits of discrete local sets are local}

\begin{lemma} \label{Hausdorfflimitislocal}
Let $D_n$ be a sequence of $\TG$-domains with maps $\phi_n:D_n
\rightarrow \H$ such that $r_{D_n} \rightarrow \infty$ as $n
\rightarrow \infty$.  Let an instance $h$ of the GFF on $\H$ be coupled
with the discrete GFF on each $D_n$, as in Section \ref{ss.GFFapprox}.
Let $A_n$ be a sequence of discrete local
subsets of $D_n \cap \TG$.
Then there is a subsequence along which
the law of $(h,\phi_n A_n)$ converges weakly (in the space of
measures on $\widehat \Omega \times \Gamma$)
to a limiting coupling
$(h, A)$ with respect to the sum of the metric $\widehat d$ on the first component
and $\dhaus$ on the second component. In any such limit,
$A$ is local.
\end{lemma}

\begin{proof}
Lemmas \ref{l.nicemetric}, \ref{l.regular}, and \ref{l.weaklimit} imply the existence of the subsequential limit
$(h, A)$, so it remains only to show that in any such limit $A$ is local.
We will prove that
characterization (\ref{i.B2cond}) for locality as given
in Lemma \ref{l.localdefequivalence} holds.  For this, it suffices to show
that for every deterministic open $B \subset \H$ and function $\phi \in -\Delta H_s(B)$
(supported in a compact subset of $B$) the law of $(h,
\phi)$ is independent of the pair $(S,\tilde A)$ (as defined in Lemma \ref{l.localdefequivalence})
together with the projection of $h$ onto $\Harm_B$.   Here we are using the fact
that for Gaussian fields the marginals characterize the field; see Lemma
\ref{p.projectionisadistribution}.  It is clearly enough to consider the case that
$B$ has compact closure in $\H$.

Fix $g \in H_s(B)$ and set $\phi = -\Delta g$.  Let $\mathbb S_n$ denote the space $\{f\circ\phi_{D_n}^{-1} : f
\in H_{\TG}(D_n) \}$.
By Lemma \ref{projectiveconvergence},
we can approximate $g$ by elements $g_n$ in $\mathbb S_n$ in such a way that
$\|g_n - g\|_\nabla \rightarrow 0$ as $n \rightarrow
\infty$.  Let $B'$ be the set of points in $B$ of distance at least $\epsilon$ from $\partial B$,
where $\epsilon$ is small enough so that $g$ is compactly supported in $B'$.
In fact, the construction given in the proof of Lemma \ref{projectiveconvergence}
ensures that each $g_n$ will be supported in $B'$ for all $n$ sufficiently large.

Now, for each fixed $n$, let $h^1_n$ denote the projection of $h$ onto the space of
functions in $\mathbb S_n$ that vanish outside of $B'$.  Let $h^2_n$ and $h^3_n$ be such that $h^1_n + h^2_n$
is the projection of $h$ onto $\mathbb S_n$ and $h^1_n + h^2_n + h^3_n = h$.  Clearly,
$h^1_n$, $h^2_n$, and $h^3_n$ are mutually independent, since they are projections of $h$
onto orthogonal spaces.

Following characterization \ref{i.Aislocalequiv2} of Lemma \ref{l.localiff}, let
$S_n$ be the event that $A_n$ includes a vertex of a triangle whose image under $\phi_n$
intersects $B'$, and let $\tilde A_n$ be equal to $A_n$
on the event $S_n^c$ and $\emptyset$ otherwise.  By Lemma \ref{l.localiff}, conditioned on $h^2_n$, the
pair $(S_n ,\tilde A_n)$ is independent of $h^1_n$.  In fact, (since
$h^3_n$ is a priori independent of the triple $(A_n, h^1_n, h^2_n)$), the pair $(S_n, \tilde A_n)$
is independent of $h^1_n + h^3_n$.

When $n$ is large enough, the space $\Harm_B$ is
orthogonal to $\mathbb S_n$.  Thus, for each sufficiently large $n$, $(h, g_n)_\nabla$ is independent
of the projection $h_{B^c}$ of $h$ onto $\Harm_B$ {\em and} the pair $(S_n, \tilde A_n)$.

Now, since $\|g_n - g\|_\nabla \rightarrow 0$ as $n \rightarrow
\infty$, the random variables $(g_n - g ,h )_\nabla$ tend to zero in law
as $n \to \infty$.  Since weak limits of independent random variables are independent,
we conclude that in any weak limit $(h,S_\text{lim},\tilde A_\text{lim}, A)$ of the
quadruple $(h, S_n,\phi_n \tilde A_n, \phi_n A_n)$
(again, subsequential limits exist by Lemmas \ref{l.nicemetric}, \ref{l.regular}, and \ref{l.weaklimit}),
the value $(h,g)_\nabla$ is independent of $h_{B^c}$ and $(S_\text{lim},\tilde A_\text{lim})$.
The event $S_\text{lim}$ contains the event $S$ (since
any Hausdorff limit of sets that intersect $B'$ must intersect $B$), and thus the pair $(S_\text{lim},
\tilde A_\text{lim})$
determines the pair $(S,A)$.  This implies that $(h, g)_\nabla$ is independent of $h_{B^c}$ and $(S, \tilde A)$.
Since this is true for all $B$ and $g$ supported on
$B$, we conclude that $A$ is local.
\qed
\end{proof}

\subsection{Statement of the height gap lemma}
We now state the special case of the {\bf height gap lemma} (as
proved in \cite{\DGFFpaper}) that is relevant to the current work.
(The lemma in \cite{\DGFFpaper} applies to more general boundary
conditions.)

As usual, we let $D$ be a $\TG$ domain with boundary conditions
$-\lambda$ on one arc $\p_-$ and $\lambda$ on a complementary
arc $\p_+$, and let $x_\p$ and $y_\p$ denote respectively the
clockwise and counterclockwise endpoints of $\p_+$.

\begin{figure}
\begin{center}
\includegraphics[width=\linewidth]{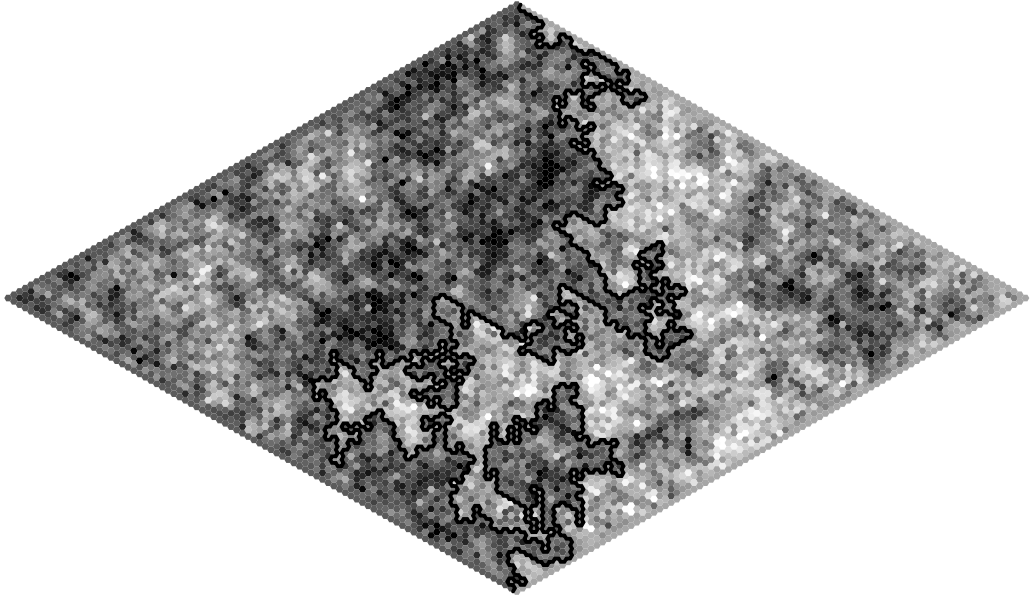}
\end{center}
\caption{\label{f.interface}Gaussian free field on faces of hexagonal lattice---faces
shaded by height---with boundary conditions equal to $-\lambda$ on
the left boundary arc and $\lambda$ on the right boundary arc, where
$\lambda>0$ is a constant. Thick line indicates chordal interface
between positive and negative height hexagons.  In the figure,
$\lambda$ is taken to be the special constant for which, as the mesh size is
taken to zero, the law of the interface converges to that of
\SLEkk4/.}
\end{figure}

Let $\gamma^T$ be the
path in the dual lattice of $\TG$ from $x_\p$ to $y_\p$ that has
adjacent to its right hand side vertices in $\p_+$ or vertices
where $h>0$ and has adjacent to its left hand side vertices in
$\p_-$ or vertices where $h<0$, stopped at some stopping time $T$
for the discrete exploration process.  This is the path that traces
the boundary between hexagons with positive sign and hexagons with
negative sign in the dual lattice, as described in \cite{\DGFFpaper}
and illustrated in Figure \ref{f.interface}. Let $\vv$ be some vertex of
$\TG$ in $D$.

Let $V_-$ denote the vertices on the left side of $\gamma^T$
together with the vertices in $\p_-$ and $V_+$ the vertices on
the right side or in $\p_+$. Let $F_T$ denote the function that is $+\lambda$
on $V_+(\gamma^T)$, $-\lambda$ on $V_-(\gamma^T)$ and
discrete-harmonic at all other vertices in $\closure D$.  Let
$h_T$ be the discrete harmonic interpolation of the values of $h$
on $V_-(\gamma^T) \cup V_+(\gamma^T)$ and on all $\TG$-vertices in
$\p D$.

\begin{lemma}\label{heightgaplemma}
For some fixed value of $\lambda>0$, we have
$$
h_T(\vv)-F_T(\vv)\to 0
$$
in probability as $T$, $D$ and $\vv$ are taken so that $\dist(\vv,\p
D)\to\infty$.  Similarly, if $\vv$ is a random vertex (with law
independent of $h$) supported on the set of points of distance at
least $r$ from $\partial D$, then as $r \to \infty$
$$\Eb{h_T(\vv)-F_T(\vv)\md\gamma^T}$$
(viewed as a random variable depending on $\gamma^T$ --- the expectation is respect
to both $\vv$ and $h_T$) tends to zero in probability.
\end{lemma}

\section{Proofs of main results} \label{proofsofmainresultssection}



\proofof{Theorem \ref{uniquecontinuumpath}}
In Section \ref{heightevolutionsection} (Lemma \ref{contourplusGFF}),
we explicitly produced a coupling of $W$ (the Loewner driving parameter of an \SLEkk4/) and $h$ (the
GFF on $\H$ with $\pm \lambda$ boundary conditions) with the conformal
Markov property described in
Theorem \ref{uniquecontinuumpath}.  Lemma \ref{lawofpathlemma} implies
that any $(\tilde h, \tilde W)$ that satisfies the hypotheses of
Theorem \ref{uniquecontinuumpath} must have this same law---and that
the value of $\lambda$ is indeed $\sqrt{\pi/8}$.

All that remains to prove in
Theorem \ref{uniquecontinuumpath}
is items \ref{i.Wdetermined} and \ref{i.samepathasincontinuumcontoutheorem}.
To prove \ref{i.Wdetermined} we must show that $W$ is equivalent (up to redefinition on a set of
measure zero) to an $\mathcal F$-measurable function from $\Omega$
to $\Lambda$.  In other words, we must show that given $h$, the conditional
law of $W$ (in any regular version of the conditional probability) is almost
surely supported on a single element of $\Lambda$.

Let $h$ be an instance of the GFF (with boundary conditions
$-\lambda$ on $(-\infty, 0)$ and $\lambda$ on $(0, \infty)$).
Write $\Phi(z) = -z^{-1}$.  Then $\Phi$ is a conformal automorphism of $\H$ sending
$0$ to $\infty$ and $\infty$ to zero, and $-h \circ \Phi$ has the same
law as $h$ (where $-h \circ \Phi$ is the pullback of $h$ as defined in Section \ref{ss.notation}).
  Let $W$ and $V$ be random elements of $\Lambda$ coupled with $h$
in such a way that
\begin{enumerate}
\item The pair $(h, W)$ satisfies the hypotheses of Theorem  \ref{uniquecontinuumpath}.
\item The pair $(-h \circ \Phi, V)$ satisfies the hypotheses of Theorem \ref{uniquecontinuumpath}.
\item Given $h$, $V$ and $W$ are independent of one another.
\end{enumerate}

Let $\gamma^1$ be the path with Loewner evolution given by $W$ and let $\gamma^2$ be
the image of the Loewner evolution generated by $V$ under $\Phi$.  Then the law of
$\gamma^1$ is that of an \SLEkk4/ from $0$ to $\infty$ and that the law of $\gamma^2$ is that
of an \SLEkk4/ from $\infty$ to $0$.

For any fixed time $T$, conditioned on $\gamma^2([0,T])$, the law of $h$ is that of a GFF
on $\H \setminus \gamma^2([0,T])$ with boundary conditions of $-\lambda$ on the left side of
$\gamma^2([0,T])$ and $(-\infty,0)$ and $\lambda$ on the right side of $\gamma^2([0,T])$ and
$(0,\infty)$.  In particular (recall Lemma \ref{l.gammaislocal}) $\gamma^2([0,T])$ is local, and
the same holds if $T$ is a stopping time of $\gamma^2([0,T])$.

If we fix a stopping time $T_1$ for $\gamma^1$, then $\gamma^1([0,T_1])$ is also local.  On the event that $\gamma^1([0,T_1])$ and $\gamma^2([0,T])$ do not intersect each other,
Lemma \ref{doublelocal} yields that the conditional law of $h$ given both sets is that of a
GFF on the complement of these sets, with the expected $\pm \lambda$ boundary conditions.  Since
the same holds for any $T_1$,
Lemma \ref{lawofpathlemma} and Lemma \ref{doublelocal} imply that
conditioned on $\gamma^2([0,T])$, the law of $\gamma^1$ --- up until the first time it hits $\gamma^2([0,T])$,
is that of an \SLEkk 4/ in $\H \setminus \gamma^2([0,T])$, started at $0$ and targeted at $\gamma^2(T)$.


It follows that almost surely $\gamma^1$ hits $\gamma^2([0,T])$ for the first
time at $\gamma^2(T)$. Since this
applies to any choice of $T$, we conclude that $\gamma^1$ hits a
dense countable set of points along $\gamma^2$, and (by symmetry) $\gamma^2$
hits a dense countable set of points along $\gamma^1$, and hence
the two paths (both of which are almost surely continuous simple
paths, by the corresponding almost sure properties of \SLEkk 4/) are equal almost
surely.  This implies that conditioned on $V$ and $h$, the law of $W$
is almost surely supported on a single element of $\Lambda$.  Since
$V$ and $W$ are conditionally independent given $h$, it follows that
conditioned on $h$, the law of $W$ is almost surely supported on a single
element of $\Lambda$. The proof of item \ref{i.samepathasincontinuumcontoutheorem}
will be established during the proof of Theorem \ref{continuumcontourtheorem} below.
\QED

In what follows, we use the notation introduced in Section \ref{ss.GFFapprox}
and the statement of Theorem \ref{continuumcontourtheorem}.
We will first give a proof of Theorem \ref{continuumcontourtheorem}
that assumes the main result
of \cite{\DGFFpaper}, namely that the discrete interfaces
converge in law to \SLEkk 4/ with respect to the metric $\dstrong$.
Afterwards, we will show how this convergence in law can be derived
from Lemma \ref{heightgaplemma} (the height gap lemma) and Theorem \ref{continuumcontourtheorem}.
That is, we give an alternate way of deriving the main
result of \cite{\DGFFpaper} (in the case of $\pm \lambda$ boundary conditions)
from Theorem \ref{continuumcontourtheorem}, so that the only result
from \cite{\DGFFpaper} that we really need for this paper is
Lemma \ref{heightgaplemma}.
(The proof of Lemma~\ref{heightgaplemma} admittedly takes about $2/3$ of the body of~\cite{\DGFFpaper},
excluding the introduction, preliminaries, etc.)

\proofof{Theorem \ref{continuumcontourtheorem}}
Let $\phi_n:D_n\to \H$ be a sequence of conformal homeomorphisms from $\TG$-domains $D_n$ to $\H$
such that $\lim_{n\to\infty}\rrr{D_n}=\infty$.
Let $\hg^n=\hg_{D_n}$ denote the image in $\H$ of the interface of the
coupled discrete GFF in $D_n$.
For each fixed $t$, by Lemma \ref{l.weaklimit}, there is a subsequence of the
$D_n$ along which the pair $(\hg^n([0,t]), h)$
converges in law
(with respect to the sum of the Hausdorff metric
on the first component and the $\widehat d$ metric on the second
component)
to the law
of a random pair $(\gamma([0,t]), h)$, where the marginal law of $\gamma$ (by the main result
of \cite{\DGFFpaper}) must be \SLEkk4/.

By Theorem \ref{uniquecontinuumpath} and Lemma \ref{l.weakgraphconvergence}, it
will be enough to show that any such limiting pair $(\gamma, h)$
satisfies the hypotheses of Theorem \ref{uniquecontinuumpath}.
For each fixed $t$, by Lemma
\ref{Hausdorfflimitislocal}, $\gamma([0,t])$ is a local set in this limiting coupling.
We next claim that $\condexp_{\gamma([0,t])}$ is almost surely given by
$$h_t:= \lambda \left( 1- 2\pi^{-1}\arg (g_t-W_t)\right),$$
where $g_t$ is the Loewner evolution, driven by $W_t$, that corresponds to $\gamma$.
Note that since $h_t$ is a bounded function that is defined almost everywhere in $\H$,
it may be also viewed as a distribution on $\H$ in the obvious way:
$(h_t, \density) = \int h_t(z) \density(z)dz$.


Once this claim is proved, Theorem \ref{continuumcontourtheorem} is
immediate from Theorem \ref{uniquecontinuumpath}, since the claim implies
that the limiting law of $(h, W)$
satisfies the hypotheses of Theorem \ref{uniquecontinuumpath} and thus $W$
is almost surely the $\Lambda$-valued function of $h$ described in
Theorem \ref{uniquecontinuumpath}---and the fact that this convergence holds for any subsequence of the
$D_n$ implies that it must hold for the entire sequence.

Let $A_n$ denote the set of vertices incident to the left or right
of the preimage of the path $\hg^n([0,t])$ in $D_n$ (so that each $A_n$ is a discrete
algorithmically local set, representing the set of vertices whose values
are observed up to the first point in the exploration algorithm that the capacity of the
image of the level line in $\H$ reaches $t$).

Let $\condexp^n_t$ denote the conditional expectation of $h$
given the values of $h_{D_n}$ on vertices in $A_n$, viewed as a distribution---more precisely,
$(\condexp^n_t, \density) = (\widehat h_{D_n} \circ
\phi_n^{-1}, \density)$, where $\widehat h_{D_n}$ is the (piecewise affine interpolation
of) the discrete harmonic interpolation to $D_n$ of the values of $h_{D_n}$ on the
vertices of $A_n$ and on the boundary vertices.
Let $W^n_t$ be the Loewner driving parameter for $\hg^n([0,t])$.
Fix $t \geq 0$ and consider now the triple:
$$(W^n_t, \condexp^n_t, \hg^n([0,t]) ).$$
By Lemma \ref{l.weaklimit}, this converges along a subsequence in law
(with respect to the $\dstrong$ metric on first coordinate plus the $\widehat d$ metric
on the second coordinate plus the Hausdorff $d_*$ metric on the third coordinate)
to a limit $(W_t, \condexp_t, K_t)$.  We may define the Loewner evolution
$g^n_t$ in terms of $W^n_t$ and analogously define
$$h^n_t:= \lambda \left( 1- 2\pi^{-1}\arg (g^n_t-W^n_t)\right).$$
For each $\density \in H_s(\H)$, we claim that the random quantity $$(h^n_t, \density) - (\condexp^n_t, \density)$$
is a continuous function of the triplet above, which implies that the difference between this quantity and $(h_t, \density) - (\condexp_t, \density)$
converges in probability to zero.  The continuity of the latter term holds simply since
$\density$ is smooth and compactly supported (and thus the $\Delta^a \density$ lies in $L^2$ for all $a$), while the former piece
is continuous with respect to the $\dstrong$ metric on $\gamma$.  Following the proof of Lemma \ref{Hausdorfflimitislocal}, it is not hard to see that
$\condexp_t=\condexp_{K_t}$, since on the discrete level, once one conditions on $h^n_t$ and
$\condexp^n_t$ the conditional law of the field minus $\condexp^n_t$ is that of a zero boundary DGFF
on the set of unobserved vertices.


A harmonic function is determined by its values
in an open set, since a harmonic function is the real part of a (possibly multi-valued) analytic function, so it is now enough to show that for each such $\density$, the
difference between the conditional
expectation of $(h, \density)$ --- given $A_n$ and the value of
$h_{D_n}$ on $A_n$ --- and the value $(h^n_t, \density)$ converges in
probability to zero as $n \rightarrow \infty$ (along a subsequence for
which a weak limit exists).

The expected value of $(h, \density)$ given the values on $A_n$
is given by $(\condexp^n_t, \density)$,
which is in turn a weighted average of the values
of $\widehat h_{D_n}$ on vertices of $D_n$ which lie on triangles that intersect the
image under $\phi_n$ of the support of $\density$.  It follows from Lemma \ref{heightgaplemma}
that if $\density$ is such that the distance between these vertices and $A_n$
necessarily tends to $\infty$ as $n \to \infty$ (which is the case if
$\density$ is compactly supported in the
complement of the set of points that can be reached by a Loewner evolution
up to time $t$), then any subsequential weak limit of the
law of the triplet above is the same as it would be
if $\widehat h_{D_n}$ were replaced by
the discrete harmonic function which is $-\lambda$ on the left-side vertices of
$A_n$ and $\lambda$ on the right side.  By standard estimates relating
discrete and continuous harmonic measure (it is enough here to recall that discrete random walk
scales to Brownian motion), we therefore
have $(h^n_t, \density) - (\condexp^n_t, \density) \to 0$ in law as $n \to \infty$ and thus
$(h_t, \density) - (\condexp_{K_t}, \density) = 0$ almost surely.  Since this is true
for any $\density$ which is necessarily supported off of $K_t$, we have
$h_t = \condexp_{K_t}$ on $\H \setminus K_t$ almost surely as desired.

\QED

We now give an alternate proof of the fact, proved in \cite{\DGFFpaper},
that the $\hg^n$ converge in law to \SLEkk4/.
Using the notation introduced in the previous proof, write
$$T^n_\epsilon(\density) = \inf \{ t: (h^n_t, \density) - (\condexp^n_t, \density) > \epsilon \}.$$
$$U^n_\epsilon(\density) = \inf \{ t: (h^n_t, \density) - (\condexp^n_t, \density) < -\epsilon \}.$$
Consider a subsequential limit along which the quadruple
$$(t, W^n_t, \condexp^n_t, \hg^n([0,t])),$$
defined with $t = T^n_\epsilon(\density)$, converges in law to a limit
$$(T, W_T, \condexp_{K_T}, K_T).$$
We claim that in such a limit, for any fixed neighborhood of the support of $\density$,
$T$ is almost surely large enough so that $K_T$ intersects that neighborhood.  The
arguments in the previous proof would imply that---on the event that $K_T$ does
does not intersect such a neighborhood---we have $$(h_T, \density) - (\condexp_{K_T}, \density) = 0$$
almost surely.  However, since $(h^n_t, \density) - (\condexp^n_t, \density) > \epsilon$ on
the analogous event for each $n$---and the values $(h^n_t, \density) - (\condexp^n_t, \density)$
converge to
$(h_T, \density) - (\condexp_{K_T}, \density)$---we must have
$$(h_T, \density) - (\condexp_{K_T}, \density) > \epsilon$$ almost surely on this event, which
implies that the event has probability zero.  A similar argument holds with $U$ in place of $T$.

The law of $(\condexp^n_t, \density)$, up until the first time that
$\hg^n([0,t])$ intersects a fixed neighborhood of the support of $\density$,
is a martingale whose largest increment size tends to zero in $n$.  Thus, if we parameterized
time by quadratic variation, then the limiting process
$(\condexp_{K_T}, \density)$ would be a Brownian motion up until the first
time that $K_T$ intersects the support of $\density$, and $(h^n_t, \density) - (\condexp^n_t, \density)$
would converge in law to this limiting process with respect to the supremum norm on finite
time intervals.

The height gap lemma implies that
$(h_T, \density) = (\condexp_{K_T}, \density)$ for any fixed stopping time (as discussed
in the previous proof).  We know that $(\condexp_{K_T}, \density)$
is a Brownian motion when parameterized by
$-E_t(\density)$, as defined in Section \ref{heightevolutionsection}, so
the same must be true for $(h_T, \density)$, and the
arguments in the proof of Lemma \ref{lawofpathlemma} then
imply that the law of $W$ is that of $\sqrt 4$ times a Brownian motion (so that
the Loewner evolution generated by $W$ is \SLEkk 4/) and $\lambda =
\sqrt{\pi/8}$.  Since the $(h^n_t, \density)$ converge uniformly
to $(h_t, \density)$, the same arguments imply that the
$W^n$ converge in law to $W$ with respect
to the supremum norm on compact intervals of time.

The fact that this driving parameter convergence holds for both forward and reverse parameterizations of the path
implies that the convergence also holds in $\dstrong$ by the main result of \cite{SheffieldSun}.

\section{Remark on other contour lines} \label{s.corsection}
A more general problem than the one dealt with in this paper is to try
to identify the collection of {\em all} chordal contour lines and contour loops
of an instance of the GFF with {\em arbitrary} boundary conditions---and to show that they
are limits of the chordal contour lines and contour loops of the piecewise linear approximations
of the field.  The second author is currently collaborating with Jason Miller on some aspects of
this general problem.

For now, we only briefly mention one reason to expect Theorems \ref{continuumcontourtheorem}
and \ref{uniquecontinuumpath} to also provide information about the general problem.
Let $h$ be a GFF with boundary conditions $h_\p$ as in Theorem \ref{continuumcontourtheorem}.
When $\psi \in H_s(D)$, it is not hard to see that the law of
$h$ is absolutely continuous with respect to the law of $h+\psi$.

\begin{corollary} \label{c.contourcor}
In the context of Theorem \ref{continuumcontourtheorem}, if $\psi \in H_s(\H)$ and
$h$ is replaced by $h+\psi$, then
as $\rr\to\infty$, the random paths $\hg_D=\phi_D\circ\gamma_D$,
viewed as $\Lambda_C$-valued random variables on
$(\Omega, \mathcal F)$, converge in probability (with respect to
the metric $\dstrong$ on $\Lambda_C$) to an $(\Omega, \mathcal
F)$-measurable random path $\gamma^\psi \in \Lambda_C$ whose law
is absolutely continuous with respect to the law of \SLEkk 4/.
\end{corollary}

We may interpret $\gamma^\psi$ as the {\bf zero contour line of $h + \psi$}.  Note that if $\psi$ is equal to a constant
$-C$ on an open set, then the intersection of $\gamma^\psi$ with this set can be viewed as (a collection of arcs of) a $C$ contour line of $h$.
By choosing different $\psi$ functions and patching together these arcs, one might hope to obtain a family of height $C$ contour loops.
It requires work to make all this precise, however, and we will not discuss this here.

\bibliographystyle{halpha}
\bibliography{mr,prep,notmr,DGFFcontours,GFFcontours}

%
%
%
\end{document}